\documentclass[11pt,oneside]{amsart}
\pdfoutput=1
\usepackage{amsmath,ifthen, amsfonts, amssymb,
srcltx, amsopn, color, mathrsfs}
\usepackage{verbatim}
\usepackage{overpic}
\usepackage[margin=0in]{caption}

\newcommand{\showcomments}{yes}

\newsavebox{\commentbox}
{\ifthenelse{\equal{\showcomments}{yes}}%
{\footnotemark
        \begin{lrbox}{\commentbox}
        \begin{minipage}[t]{1.25in}\raggedright\sffamily\tiny
        \footnotemark[\arabic{footnote}]}
{\begin{lrbox}{\commentbox}}}%
{\ifthenelse{\equal{\showcomments}{yes}}%
{\end{minipage}\end{lrbox}\marginpar{\usebox{\commentbox}}}
{\end{lrbox}}}

\newtheorem{thm}{Theorem}[section]
\newtheorem{lem}[thm]{Lemma}

\newtheorem{cor}[thm]{Corollary}

\newtheorem{prop}[thm]{Proposition}
\newtheorem*{thmA}{Theorem~A}
\newtheorem*{thmB}{Theorem~B}
\newtheorem*{corC}{Corollary~C}

\theoremstyle{definition}
\newtheorem{defn}[thm]{Definition}
\newtheorem{rem}[thm]{Remark}
\newtheorem{exmp}[thm]{Example}

\newtheorem{claim*}{Claim}

\DeclareMathOperator{\Aut}{Aut}

\DeclareMathOperator{\stabilizer}{Stab}

\DeclareMathOperator{\linspan}{Span}

\newcommand{\field}[1]{\mathbb{#1}}
\newcommand{\integers}{\ensuremath{\field{Z}}}

\newcommand{\naturals}{\ensuremath{\field{N}}}
\newcommand{\reals}{\ensuremath{\field{R}}}
\newcommand{\Euclidean}{\ensuremath{\field{E}}}

\makeatletter

\newcommand{\Rmnum}[1]{\mathbf{{\expandafter\@slowromancap\romannumeral #1@}}}

\newcommand{\simp}{\ensuremath{\partial_{_{\vartriangle}}}}

\makeatother

\let\oldmarginpar\marginpar
\renewcommand\marginpar[1]{\-\oldmarginpar[\raggedleft\footnotesize #1]%
{\raggedright\footnotesize #1}}

\begin{document}
\title[Cubulated crystallographic groups]{Cocompactly cubulated crystallographic groups}
\author[Mark~F.~Hagen]{Mark F. Hagen}
           \address{Dept. of Math.\\
                    University of Michigan\\
                    East Hall, 530 Church St.\\
                    Ann Arbor, MI, 48109 USA}
           \email{markfhagen@gmail.com}

\date{\today}
\maketitle

\begin{abstract}
We prove that the simplicial boundary of a CAT(0) cube complex admitting a proper, cocompact action by a virtually $\integers^n$ group is isomorphic to the hyperoctahedral triangulation of $S^{n-1}$, providing a class of groups $G$ for which the simplicial boundary of a $G$-cocompact cube complex depends only on $G$.  We also use this result to show that the cocompactly cubulated crystallographic groups in dimension $n$ are precisely those that are \emph{hyperoctahedral}.  We apply this result to answer a question of Wise on cocompactly cubulating virtually free abelian groups.
\end{abstract}

\section{Introduction}\label{sec:introduction}
In this paper, we use the notion of the \emph{simplicial boundary} of a CAT(0) cube complex to study actions of crystallographic groups on CAT(0) cube complexes.  For $n\geq 1$, an \emph{$n$-dimensional crystallographic group} $G$ is a discrete subgroup of the Euclidean group $\reals^n\rtimes O(n,\reals)$ that acts properly and cocompactly by isometries on $\Euclidean^n$.  Bieberbach's theorems~\cite{Bieberbach1911,Bieberbach1912} tell us that there is an exact sequence
\[1\rightarrow T_{_G}\rightarrow G\stackrel{\psi}{\rightarrow} P_{_G}\rightarrow 1,\]
where the \emph{translation subgroup} $T_{_G}=G\cap\reals^n$ and the \emph{point group} (or \emph{holonomy group}) $P_{_G}$ is a finite subgroup of $O(n,\reals)$.  Moreover, $T_{_G}$ is the unique maximal abelian normal subgroup of $G$.  Bieberbach showed that, for any $n$, there are finitely many isomorphism classes of $n$-dimensional crystallographic groups and, up to conjugation by affine transformations, each crystallographic group acts in a unique way on $\Euclidean^n$.  Conversely, extensions of $\integers^n$ by finite groups acting faithfully by isometries on $\Euclidean^n$ are crystallographic groups~\cite{Zassenhaus}.  Because of Zassenhaus's result, cubulations of crystallographic groups are closely related to the more general question of the possibility of cocompactly cubulating virtually free abelian groups, and several of our conclusions about crystallographic groups apply in this more general context.

Our main goal is to characterize the crystallographic groups that a admit proper, cocompact action on a CAT(0) cube complex $\mathbf X$ by describing the \emph{simplicial boundary} $\simp\mathbf X$ of $\mathbf X$.  We then examine the action of $G$ on the simplicial complex $\simp\mathbf X$ to obtain a description of the possible point groups.  Conversely, there is a standard cubulation of crystallographic groups and, if the point group is of one of the admissible types, then this cubulation is cocompact.

\subsection{Hyperoctahedral boundary}
The simplicial boundary of a CAT(0) cube complex, introduced in~\cite{HagenBoundary}, is an invariant of the 1-skeleton, encoding non-hyperbolic behavior, and has some features in common with the Tits boundary of a CAT(0) space.  An action on a cube complex always induces an action on the simplicial boundary, but it is unknown, in general, when $\simp\mathbf X$ is a quasi-isometry invariant of $\mathbf X^1$.  In particular, we have the following problem:

\begin{center}
\emph{For which groups $G$ is it true that $\simp\mathbf X$ is isomorphic to $\simp\mathbf Y$ for any two CAT(0) cube complexes $\mathbf X$ and $\mathbf Y$ on which $G$ acts properly and cocompactly?}\\
\end{center}

Our first result, Theorem~\ref{thm:boundaryoctahedron}, solves this problem for all virtually-$\integers^n$ groups:

\begin{thmA}\label{thm:thmA}
Let $n\geq 1$.  If the CAT(0) cube complex $\mathbf X$ admits a proper, cocompact action
by a virtually-$\integers^n$ group, then $\simp\mathbf X$ is isomorphic to the $(n-1)$-dimensional hyperoctahedron $\mathbf Q_n$.
\end{thmA}

Theorem~A is related to Theorem~\ref{thm:product_action}, which says that if $V$ is a virtually-$\integers^n$ group, then $V$ is cocompactly cubulated if and only if $V$ acts properly and cocompactly on $\mathbf R_n$, the standard tiling of $\Euclidean^n$ by $n$-cubes.  This assertion appears as Lemma~16.12 in~\cite{WiseIsraelHierarchy}, where Wise deduces it from the Flat Torus Theorem~\cite{BridsonHaefliger}.  We give an alternative proof, deducing it from Theorem~A using results of Caprace-Sageev~\cite{CapraceSageev} and~\cite{HagenBoundary}.

Although Theorem~A applies to a rather specific class of groups and cube complexes, we present a proof which seems amenable to generalization to the situation in which $G$ is a group acting properly and cocompactly on two distinct cube complexes, and indeed it seems a similar approach may answer the above question positively for many cocompactly cubulated groups.  In Section~\ref{sec:bounarychar}, we also sketch a quick proof of Theorem~A that uses the somewhat heavy machinery of~\cite{CapraceSageev}; however, our actual argument proceeds directly from the definition of the simplicial boundary.

\subsection{Cocompactly cubulated crystallographic groups}
The $n$-dimensional crystallographic group $G$ is \emph{hyperoctahedral}, in a sense made precise in Section~\ref{sec:background}, if $P_{_G}$ injects into $\Aut(\mathbf Q_n)\cong
O(n,\integers)$ in a way that is consistent with the action of $P_{_G}$ on $\Euclidean^n$ induced by $\theta$.  Theorem~\ref{thm:main}, which we deduce from Theorem~\ref{thm:boundaryoctahedron}, is:

\begin{thmB}\label{thmB}
The following are equivalent, for an $n$-dimensional crystallographic group $G$:
\begin{enumerate}
\item $G$ is hyperoctahedral.
\item $G$ acts properly and cocompactly on a CAT(0) cube complex.
\end{enumerate}
\end{thmB}

The conclusion of Theorem~\ref{thm:product_action} holds for crystallographic groups by combining Theorem~\ref{thm:main} with Theorem~\ref{thm:cubulatingcrystallographic}: the former says that a cocompactly cubulated crystallographic group $G$ is hyperoctahedral, after which the latter provides a proper, cocompact action of $G$ on $\mathbf R_n$.  Thus, in the special case of crystallographic groups, Theorem~\ref{thm:product_action} has a proof that sidesteps some of the machinery used in the proof below or in~\cite{WiseIsraelHierarchy}.

\begin{figure}
  \includegraphics[width=0.8\textwidth]{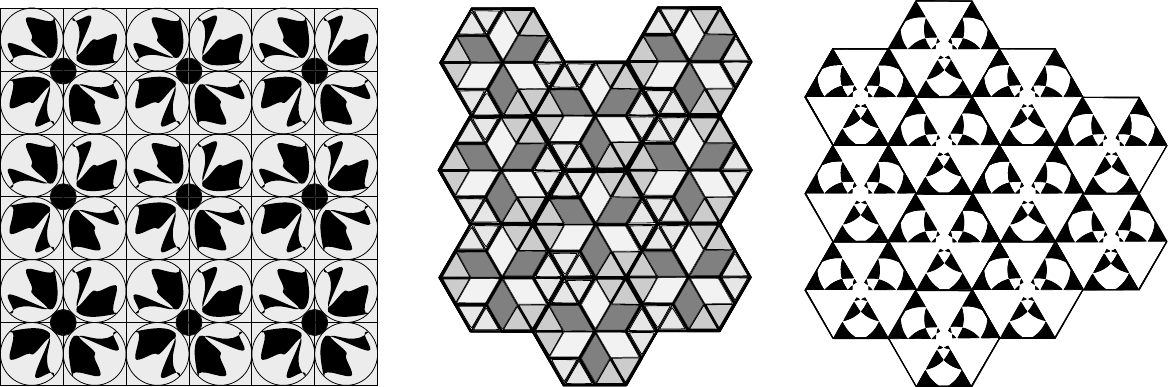}\\
  \caption{Parts of three different periodic tilings of $\Euclidean^2$.  The automorphism group of the tiling shown at left acts properly and cocompactly on $\mathbf R_2$, but the automorphism groups of the center and right tilings do not act properly and cocompactly on cube complexes.  They do, however, act properly on $\mathbf R_3$.}\label{fig:truncatedsquare}
\end{figure}

The proof of Theorem~\ref{thm:cubulatingcrystallographic} is an application of the version of Sageev's construction of a $G$-cube-complex discussed in~\cite{HruskaWiseAxioms}.  To cubulate a crystallographic group $G$, one uses a natural collection of \emph{geometric walls}, which in this case are codimension-1 affine subspaces of $\Euclidean^n$.  As explained in Section~\ref{sec:cubulating}, one always obtains a proper action of the $n$-dimensional crystallographic group $G$ on $\mathbf R_N$ for some $N\geq n$ (see also Section~16 of~\cite{WiseIsraelHierarchy}); if $G$ is hyperoctahedral, then we find we can take $N=n$, so that this ``standard cubulation'' of $G$ is cocompact.

\begin{rem}\label{rem:coxeter_cubes}
The above cubulation is related to that for Coxeter groups constructed by Niblo-Reeves~\cite{NibloReeves:coxeter_cubes}.  Williams gave a condition~\cite{WilliamsCoxeter} on the Coxeter group guaranteeing cocompactness of the latter cubulation; Caprace-M\"{u}hlherr~\cite{CapraceMuhlherr} and, independently, Bahls~\cite{Bahls} provided a more easily verified condition on a Coxeter group that is equivalent to cocompactness of the Niblo-Reeves cubulation.  For crystallographic groups, the existence of a cocompact cubulation is equivalent to cocompactness of the ``standard cubulation'', by Theorems~\ref{thm:main} and Theorem~\ref{thm:cubulatingcrystallographic}, or by Theorem~\ref{thm:product_action}.  Since the Niblo-Reeves cubulation coincides with the standard cubulation for crystallographic Coxeter groups, it is natural to ask whether the existence of a cocompact cubulation of a Coxeter group is equivalence to cocompactness of the Niblo-Reeves complex and whether this question be approached by studying the boundary of the Niblo-Reeves complex.
\end{rem}

\subsection{Obtaining cocompactness by adding dimensions}  The issue of cubulating virtually-$\integers^n$ groups is raised in~\cite{WiseIsraelHierarchy}.  Wise's consideration of actions of virtually free abelian groups on cube complexes arose from questions about \emph{sparse} cube complexes, and the connection to crystallographic groups comes from an example, due to Dunbar, of a torsion-free 3-dimensional crystallographic group that does not act properly and cocompactly on a CAT(0) cube complex.  This example is discussed in Section~16 of~\cite{WiseIsraelHierarchy}, wherein it is also shown that, if $V$ is virtually $\integers^n$ and torsion-free, then there exists a virtually free abelian group $\ddot V$ containing $V$ and an integer $N\geq n$ such that $\ddot V$ acts properly and cocompactly on $\mathbf R_N$.  In view of this fact, Wise asked:

\begin{center}
\emph{For any virtually $\integers^n$ group $V$, does there exist $m$ such that $\integers^m\times V$ is cocompactly cubulated?}
\end{center}

We answer this question negatively in Example~\ref{exmp:trianglegroup}, using
Theorem~B and the standard cubulation.  More specifically, we show in that example that for all $m\geq 0$, the group $\integers^m\times W$ is not cocompactly cubulated, where $$W\cong\langle a,b,c\mid[a,b],c^6,cac^{-1}=b,cbc^{-1}=a^{-1}b\rangle\cong\integers^2\rtimes\integers_6.$$  The main thrust of this example is that a cocompact cubulation of $\integers^m\times W$ would yield, via Theorem~B, an action of $W$ on a 3-cube with $c$ acting as a 6-fold rotation; this is impossible.  We also answer Wise's question negatively for the Dunbar example discussed in~\cite[Example~16.11]{WiseIsraelHierarchy}.  On the other hand, we obtain the following as a consequence of Corollary~\ref{cor:d1}, thus answering a weaker version of Wise's question affirmatively:

\begin{corC}\label{corD}
Let $G$ be an $n$-dimensional crystallographic group.  Then there exists $N\geq n$ such that $\integers^m\rtimes G$ is a cocompactly cubulated crystallographic group for all $m\geq N-n$ and a suitably-defined action of $G$ on $\integers^m$.
\end{corC}

In the case of the group $W$ discussed above, we give an action of $W$ on $\integers$ so that the resulting group $\integers\rtimes W$ acts properly and cocompactly on $\mathbf R_3$.
\subsection{Plan of the paper}
Section~\ref{sec:background} discusses hyperoctahedra, CAT(0) cube complexes, and groups acting on these objects, and also reviews the notion of the cube complex dual to a geometric wallspace and the linear separation property that guarantees properness of the action on the dual cube complex.  Section~\ref{sec:background} also contains a self-contained description of the simplicial boundary of a cube complex.  Section~\ref{sec:bounarychar} is devoted to the proof of Theorem~A, and in Section~\ref{sec:actiononboundary}, we prove Theorem~\ref{thm:main}, which establishes that cocompactly cubulated crystallographic groups are hyperoctahedral.  In Section~\ref{sec:cubulating}, we cocompactly cubulate hyperoctahedral groups and discuss Corollary~C.

\subsection*{Acknowledgments.}  I am grateful to Dani Wise for reading through this paper with me and giving valuable criticism, as well as for suggesting this problem and sharing the preprint~\cite{WiseIsraelHierarchy}, which contains, among many other things, the discussion that motivated this work.  I am also grateful to the anonymous referee for their careful reading and valuable comments and corrections.  This work was supported by a Schulich Graduate Fellowship through McGill University.

\section{Preliminaries}\label{sec:background}

\subsection{Hyperoctahedral groups}\label{sec:hyperoctahedral}

\begin{defn}[Hyperoctahedron]\label{defn:hyperoctahedron}
The \emph{0-dimensional hyperoctahedron} $\mathbf Q_1$ is the simplicial complex consisting of two nonadjacent 0-simplices.  For $n\geq 1$, the \emph{$n$-dimensional hyperoctahedron} $\mathbf Q_{n+1}$ is the simplicial join of $\mathbf Q_n$ and $\mathbf Q_1$, i.e. $\mathbf Q_{n+1}=\mathbf Q_n\star\mathbf Q_1$. 
Here, the \emph{simiplicial join} of the flag complexes $A,B$ is the flag complex $A\star B$ determined by the join of the graphs $A^1$ and $B^1$.  (Recall that a \emph{flag complex} is a simplicial complex in which any $n+1$ pairwise-adjacent 0-simplices span an $n$-simplex, and that each simplicial graph is the 1-skeleton of a unique flag complex.)  Note that $\mathbf Q_n$ is the link of a 0-cube in $\mathbf R_n$ and is thus the complex dual to the boundary of an $n$-cube.
\end{defn}

It is easily seen from the latter characterization of $\mathbf Q_n$ that $\Aut(\mathbf Q_n)$ is the automorphism group of an $n$-cube.  It follows that $\Aut(\mathbf Q_n)$ is isomorphic to the wreath product $\integers_2\wr S_n$, i.e. to $\integers_2^n\rtimes S_n$, where $S_n$ acts on $\integers_2^n$ by permuting the factors.  From this characterization, one shows that $\Aut(\mathbf Q_n)\cong O(n,\integers)$, the group of orthogonal matrices whose entries are 0 or $\pm1$.  A lucid survey of the representation theory of $\Aut(\mathbf Q_n)$ was provided by Baake in~\cite{Baake}, which contains the details of this and related representations.

\subsubsection{Hyperoctahedral crystallographic groups}
We view $\Aut(\mathbf Q_n)$ as the group $O(n,\integers)$ of $n\times n$ signed permutation matrices, i.e. the group of orthogonal matrices with integer entries, which acts by permutations on $\{\pm\dot e_i\}_{i=1}^n$, where $\{\dot e_i\}_{i=1}^n$ is the standard basis of $\reals^n$.

Now, let $G$ be an $n$-dimensional crystallographic group and let $\theta:G\rightarrow\reals^n\rtimes O(n,\reals)$ be the given faithful, proper, cocompact action on $\Euclidean^n$.  Let $\bar{\theta}:P_{_G}\rightarrow O(n,\reals)$ be the induced faithful action of $P_{_G}$, so that for each $r\in\Euclidean^n$ and $g\in G$, there exists a vector $\tau_g$ with
\[\theta(g)(r)=\bar{\theta}(\psi(g))(r)+\tau_g.\]
We denote by $t_g\in T_{_G}$ the translation along $\tau_g$.
Let $t_1,\ldots,t_n$ be a set of generators of $T_{_G}$, and for $1\leq i\leq n$, let $\dot t_i$ denote the translation vector corresponding to $t_i$.  Let $\mathcal L=\integers[\dot t_1,\ldots,\dot t_n]$ be the lattice of translations, to that $\tau_g\in\mathcal L$ and, for all $\ell\in \mathcal L$ and all $g\in G$, we have $\bar{\theta}(\psi(g))(\ell)\in\mathcal L$, i.e. $P_{_G}$ preserves the lattice.

\begin{defn}
The $n$-dimensional crystallographic group $G$ is \emph{hyperoctahedral} if there are monomorphisms $\iota:P_{_G}\rightarrow O(n,\integers)$ and $\rho:O(n,\integers)\rightarrow O(n,\reals)$ such that $\rho\circ\iota=\bar{\theta}$ and $\rho$ corresponds to conjugation by some $A\in GL(n,\reals)$, i.e. for all $p\in P_{_G}$, we have $\bar{\theta}(p)=A\iota(p)A^{-1}$.
\end{defn}

\begin{lem}\label{lem:hyperoctahedralbasis}
Let $G$ be an $n$-dimensional hyperoctahedral crystallographic group.  Then $\Euclidean^n$ has a basis $\{\dot t\}_{i=1}^n$ such that $\bar{\theta}(P_{_G})$ acts by permutations on $\{\pm\dot t_i\}_{i=1}^n$.
\end{lem}

\begin{proof}
We have hypothesized a faithful action $\iota:P_{_G}\rightarrow O(n,\integers)$ by permutations on $\{\pm\dot e_i\}_{i=1}^n$. Moreover, there exists $A\in GL(n,\reals)$ such that for all $p\in P_{_G}$, we have $A\iota(p)A^{-1}=\bar{\theta}(p)$.  Let $\dot t_i=A\dot e_i$.  Then $\bar{\theta}(p)(\dot t_i)=A\iota(p)(\dot e_i)=\pm\dot t_j$ for some $j\leq n$, so that $\bar{\theta}(P_{_G})$ acts on $\{\pm\dot t_i\}_{i=1}^n$ by permutations.
\end{proof}

\subsection{Cube complexes}\label{sec:cubecomplexes}
A CAT(0) cube complex $\mathbf X$ is a simply connected CW-complex built from unit cubes of various dimensions, in such a way that distinct cubes intersect in a common face or in the empty set, subject to the additional constraint that the link of each 0-cube of $\mathbf X$ is a flag complex.  A \emph{hyperplane} $H$ of $\mathbf X$ is a connected subspace that intersects each cube $c\cong[-\frac{1}{2},\frac{1}{2}]^d$ either in the empty set or in a subspace obtained by restricting exactly one coordinate of $c$ to 0.  The \emph{carrier} $N(H)$ of $H$ is the union of all closed cubes $c$ with $H\cap c\neq\emptyset$ and is a CAT(0) cube complex isomorphic to $H\times[-\frac{1}{2},\frac{1}{2}]$; likewise, $H$ is a CAT(0) cube complex of dimension strictly lower than that of $\mathbf X$, if $\mathbf X$ is finite-dimensional.

The hyperplane $H$ is also globally separating: $\mathbf X-H$ has exactly two components, $\mathfrak h(H)$ and $\mathfrak h^*(H)$, called \emph{halfspaces}.  The distinct hyperplanes $H,H'$ \emph{cross} if each of the four \emph{quarterspaces} $\mathfrak h(H)\cap\mathfrak h(H'),\mathfrak h(H)\cap\mathfrak h^*(H'),\mathfrak h^*(H)\cap\mathfrak h(H'),\mathfrak h^*(H)\cap\mathfrak h^*(H')$ is nonempty.  Note that $H$ and $H'$ cross if and only if $H\cap H'\neq\emptyset$ and that if $H$ and $H'$ cross, then $H\cap H'$ is a hyperplane of $H$ and of $H'$.  The above facts were proved independently in~\cite{Sageev95,Chepoi2000}.

The subspaces $A,B\subset\mathbf X$ are \emph{separated} by the hyperplane $H$ if there is a halfspace $\mathfrak h\in\{\mathfrak h(H),\mathfrak h^*(H)\}$ such that $A\subset\mathfrak h$ and $B\subset\mathbf X-\mathfrak h$.  The 1-cube $c$ is \emph{dual} to $H$ if the 0-cubes of $c$ are separated by $H$ or, equivalently, if $H\cap c$ is the midpoint of $c$.  More generally, if $x,y\in\mathbf X^0$, then the number of hyperplanes separating $x$ from $y$ coincides with the distance from $x$ to $y$ in the graph $\mathbf X^1$. In~\cite{HaglundSemisimple}, Haglund showed that the path-metric on $\mathbf X^1$ extends to a metric $d_{\mathbf X}$ on $\mathbf X$, whose restriction to each cube is the $\ell^1$ metric.  We shall always use $d_{\mathbf X}$ instead of the CAT(0) metric discussed in~\cite{BridsonThesis,Gromov87,LearyInfiniteCubes,MoussongThesis}, and in fact shall almost always consider paths in the 1-skeleton, occasionally using the fact that $\mathbf X^1$ is a median graph (see~\cite{Chepoi2000,
EppsteinFalmagneOvchinnikov,ImKl,Roller98,vandeVel_book}).

A subcomplex $Y\subseteq\mathbf X$ is isometrically embedded if and only if $Y\cap H$ is connected for each hyperplane $H$ (this is well-known; see e.g.~\cite[Section~2]{HagenQuasiArb} for the usual proof using disc diagrams).  A \emph{combinatorial interval} $I$ is the tiling by unit-length 1-cubes of a subinterval of $\reals$ whose endpoints, if any, are integers, and a \emph{combinatorial path} in $\mathbf X$ is a map $\gamma:I\rightarrow\mathbf X^1$ that sends 0-cubes to 0-cubes and 1-cubes homeomorphically to 1-cubes.  The map $\gamma$ is therefore an isometric embedding if the map that assigns to each 1-cube of the image of $\gamma$ its dual hyperplane is injective.  In such a case, $\gamma$ is a \emph{combinatorial geodesic segment} if $I$ is finite, a \emph{combinatorial geodesic ray} if $I\cong[0,\infty)$, and a \emph{(bi-infinite) combinatorial geodesic} if $I\cong\reals$.  In each of these cases, we also use the notation $\gamma$ to mean the image of the map $\gamma:I\rightarrow\mathbf X$.  If $I$ is unbounded, we write, e.g., $\gamma:\reals\rightarrow\mathbf X$ with the understanding that $\gamma$ takes integers to 0-cubes and intervals $[k,k+1],\,k\in\integers$ isometrically to 1-cubes.  We say that $\gamma$ \emph{crosses} the hyperplane $H$ or that $H$ \emph{crosses} $\gamma$ to mean that the geodesic path $\gamma$ contains a 1-cube dual to $H$.  More generally, the hyperplane $H$ \emph{crosses} the isometrically embedded subcomplex $Y\subseteq\mathbf X$ if $H\cap Y\neq\emptyset$.  In this case, $Y-Y\cap H$ has exactly two components, namely the intersections of $Y$ with the two halfspaces in $\mathbf X$ associated to $H$.

The subcomplex $Y\subseteq\mathbf X$ is \emph{convex} if, for any concatenation $ef$ of 1-cubes of $Y$ that lie on the boundary path of a (closed) 2-cube $s$ of $\mathbf X$, the 2-cube $s$ belongs to $Y$ and, more generally, if $c$ is a cube of $\mathbf X$ with a corner in $Y$, then $c\subseteq Y$; convex subcomplexes are therefore CAT(0).  If $Y$ is convex, then $Y^1$ is a convex
subgraph of $\mathbf X^1$ (or, equivalently, if $Y$ is convex with respect to
the metric $d_{\mathbf X}$), i.e. every combinatorial geodesic segment with
two endpoints in $Y^1$ is contained in $Y^1$. However, there is no ambiguity in
simply using the term ``convex'' to refer to a subcomplex, since $Y$ is convex in the above sense if and only if it is convex with respect to the CAT(0) metric~\cite{HaglundSemisimple}. We shall use the fact, proved
in~\cite{Chepoi2000} and~\cite{Sageev95}, that the carrier of any hyperplane is
a convex subcomplex.  For a more detailed account of the basic properties of
cube complexes, we refer the reader to, for example,
~\cite{BandeltChepoi_survey,Chepoi2000,HaglundSemisimple,Sageev95,
WiseIsraelHierarchy}.

\subsection{Cubical isometries}\label{sec:cubicalisometries}
Isometries of $\mathbf X$ were classified in~\cite{HaglundSemisimple}.  Let $G$
act on $\mathbf X$ and let $g\in G$.  Then either $g$ stabilizes a cube of
$\mathbf X$, in which case we say that $g$ is \emph{elliptic}, or there is a
$g$-invariant combinatorial geodesic $\alpha:\reals\rightarrow\mathbf X^1$,
called a \emph{(combinatorial) axis} for $g$, on which $g$ acts as a
translation, in which case $g$ is \emph{hyperbolic}, or there exists a
hyperplane $H$ such that $g^kH=H$ and $g^k\mathfrak h(H)=\mathfrak h^*(H)$ for
some $k>0$. If $\mathbf X$ is finite-dimensional, the last
circumstance implies that there exists $n>0$ such that $g^n$ is either elliptic
or hyperbolic.  Also, it is well-known (and readily verified from the definition of a hyperplane) that if $H$ is a hyperplane, then $gH$ is again a hyperplane.

The hyperplane $H$ is \emph{$G$-essential} if, for
each $x\in\mathbf X^0$ and each $n\geq 0$, there exist $g,g^*\in G$ such that
$gx\in\mathfrak h(H),g^*x\in\mathfrak h^*(H),$ and $\min\left\{d_{\mathbf
X}(gx,N(H)),d_{\mathbf X}(g^*x,N(H))\right\}\geq n.$
$G$ acts \emph{essentially} on $\mathbf X$ if each hyperplane is $G$-essential.  If the infinite group $G$ acts properly and cocompactly on $\mathbf X$, then there is a convex, $G$-invariant subcomplex $\mathbf Y\subseteq\mathbf X$ on which $G$ acts essentially and cocompactly; this is the \emph{essential core theorem} of~\cite{CapraceSageev}, and each hyperplane of $\mathbf X$ crossing the \emph{essential core} $\mathbf Y$ is $G$-essential.

\subsection{The cube complex dual to a wallspace}\label{sec:wallspace}
The set-theoretic notion of a \emph{wallspace} is due to Haglund-Paulin~\cite{HaglundPaulin98}.  There are various accounts of the duality between CAT(0) cube complexes and wallspaces; we refer the reader to~\cite{ChatterjiNiblo04,NicaCubulating04}.  The procedure of passing from a group action on a wallspace to an action on the dual cube complex generalizes Sageev's construction in~\cite{Sageev95}.  At present, however, we use a slightly restricted version of the language of \emph{geometric wallspaces} from~\cite{HruskaWiseAxioms}.

Let $(M,d)$ be a metric space.  A \emph{geometric wall} $W\subset M$ is a subspace such that $M-W$ has exactly two nonempty connected components $\mathfrak h(W),\mathfrak h^*(W)$, called \emph{halfspaces}, and the wall $W$ \emph{separates} $p,q\in M$ if $p$ and $q$ lie in distinct halfspaces associated to $W$.  A \emph{geometric wallspace} $\left(M,\mathcal W\right)$ consists of a metric space $M$, together with a collection $\mathcal W$ of walls such that $\#(p,q)<\infty$ for all $p,q\in M$, where $\#(p,q)$ is the number of walls in $\mathcal W$ separating $M$.  An \emph{orientation} is an assignment $\mathcal W\ni W\mapsto x(W)\in\{\mathfrak h(W),\mathfrak h^*(W)\}$
of a halfspace to each wall.  The orientation $x$ is \emph{consistent} if $x(W)\cap x(W')\neq\emptyset$ for all $W,W'\in\mathcal W$ and \emph{canonical} if for all $p\in M$ and all but finitely many $W\in\mathcal W$, we have $p\in x(W)$.  By associating a 0-cube to each consistent, canonical orientation, with 0-cubes $x$ and $y$ adjacent if and only if the corresponding orientations differ on a single hyperplane, we obtain a median graph, which is the 1-skeleton of a uniquely determined CAT(0) cube complex called the \emph{cube complex dual to} the wallspace.

Suppose the group $G$ acts by isometries on $M$, and $\mathcal W$ is $G$-invariant in the sense that $gW\in\mathcal W$ for each geometric wall $W$ and each $g\in G$.  Then $G$ acts on the dual cube complex $\mathbf X$.  The collection $\mathcal W$ of walls satisfies the \emph{linear separation property} if there exist constants $K_1,K_2$ such that for all $p,q\in M$, $d(p,q)\leq K_1\#(p,q)+K_2,$ and it is shown in~\cite{HruskaWiseAxioms} that, if $G$ acts metrically properly on $M$ and $\mathcal W$ satisfies the linear separation property, then $G$ acts properly on $\mathbf X$.  In our situation, $\mathbf X$ is always locally finite, so that $G$ acts metrically properly on $\mathbf X$ if and only if the stabilizer of each cube is finite.

\begin{rem}\label{rem:cubesasorientations}
Each hyperplane $H$ in the CAT(0) cube complex $\mathbf X$ is a geometric wall whose complementary components are the halfspaces $\mathfrak h(H),\mathfrak h^*(H)$.  It is not hard to see that the cube complex dual to the wallspace whose underlying set is $\mathbf X$ and whose walls correspond in this manner to the hyperplanes is none other than $\mathbf X$.  Sometimes, it is useful to view the 0-cubes of $\mathbf X$ as consistent, canonical orientations of the hyperplanes in $\mathbf X$, in order to construct isometrically embedded subcomplexes, as in the proof of Lemma~\ref{lem:raybuild2} and that of Lemma~\ref{lem:induction1}.
\end{rem}

\subsection{The simplicial boundary of a cube complex}\label{sec:boundary}
The simplicial boundary $\simp\mathbf X$ of the locally-finite CAT(0) cube complex $\mathbf X$ containing no infinite family of pairwise-crossing hyperplanes was introduced in~\cite{HagenBoundary}.  Since the cube complexes considered here admit proper, cocompact group actions, we can use a more concrete definition of $\simp\mathbf X$ than that in~\cite{HagenBoundary} and give an almost completely self-contained account, suited to our purposes.  More precisely, $\simp\mathbf X$ is defined in~\cite{HagenBoundary} to be a complex constructed from simplices corresponding to infinite, inseparable, unidirectional sets of hyperplanes that contain no facing triple.  Such sets of hyperplanes are modeled on the set of hyperplanes crossing a combinatorial geodesic ray, but in some (non-cocompact) situations, there are such sets for which there is no corresponding ray.  In the present paper, we define simplices at infinity in terms of rays only.

Let $\mathcal W$ be the set of hyperplanes of $\mathbf X$, and for each isometrically embedded subcomplex $A\subseteq\mathbf X$, denote by $\mathcal W(A)$ the set of hyperplanes that cross $A$, i.e. those hyperplanes $H$ such that $\mathfrak h(H)\cap A\neq\emptyset$ and $\mathfrak h^*(H)\cap A\neq\emptyset$.  Let $\gamma,\gamma':[0,\infty)\rightarrow\mathbf X$ be combinatorial geodesic rays.  If $\mathcal W(\gamma)-\mathcal W(\gamma)\cap\mathcal W(\gamma')$ is finite, i.e. if all but finitely many hyperplanes that cross $\gamma$ also cross $\gamma'$, then $\gamma'$ \emph{consumes} $\gamma$.  If $\gamma'$ consumes $\gamma$ and $\gamma$ consumes $\gamma'$, then $\gamma$ and $\gamma'$ are \emph{almost-equivalent}.  Almost-equivalence is an equivalence relation on the set $\mathfrak R\mathbf X$ of combinatorial geodesic rays in $\mathbf X$; the class represented by $\gamma$ is denoted $[\gamma]$.

\begin{exmp}[Consumption]\label{exmp:consumes}
Consider $\mathbf R_2\cong\mathbf R_1\times\mathbf R_1$.  The hyperplanes are all isomorphic to $\mathbf R_1$ and have the form $V_n=\{n+\frac{1}{2}\}\times\mathbf R_1$ or $H_n=\mathbf R_1\times\{n+\frac{1}{2}\}$.  Let $\gamma$ be the combinatorial geodesic ray whose 0-skeleton is $\{(n,n)\,:\, n\geq 0\}$ and let $\gamma'$ be the combinatorial geodesic ray whose 0-skeleton is $\{(n,0)\,:\, n\geq -4\}$.  Then $\mathcal W(\gamma')=\{V_n\,:\, n\geq-4\}$ and $\mathcal W(\gamma)=\{V_n,H_n\,:\, n\geq 0\}$.  Hence $\gamma$ consumes $\gamma'$.
\end{exmp}

The following lemma from~\cite{HagenBoundary} is used freely throughout this paper.  A set $\mathcal W'$ of hyperplanes is \emph{inseparable} if for all $W,W'\in\mathcal W'$, if a hyperplane $U$ separates $W$ and $W'$, then $U\in\mathcal W'$.

\begin{lem}\label{lem:folding}
For any $x\in\mathbf X$, and for any $[\gamma]\in\mathfrak R\mathbf X$, there exists a combinatorial geodesic ray $\gamma':[0,\infty)\rightarrow\mathbf X$ such that $\gamma'(0)=x$ and $[\gamma]=[\gamma']$.

If $\mathcal W'\subseteq\mathcal W(\gamma)$ is an inseparable subset such that $\mathcal W(\gamma)-\mathcal W'$ is finite, then there exists a combinatorial geodesic ray $\gamma'$ such that $[\gamma']=[\gamma]$ and $\mathcal W(\gamma')=\mathcal W'$.
\end{lem}

Given $[\gamma],[\gamma']\in\mathfrak R\mathbf X$, write $[\gamma]\leq[\gamma']$
if some (and hence every) representative of $\gamma'$ consumes some (and hence
every) representative of $\gamma$.  From the definition, it follows that $\leq$
partially orders $\mathfrak R\mathbf X$.  The almost-equivalence class
$[\gamma]$ is \emph{minimal} if, for each geodesic ray $\gamma'$ consumed by
$\gamma$, we have $[\gamma']=[\gamma]$.  Moreover, if $\beta,\gamma$ are
combinatorial geodesic rays, then either $\mathcal W(\gamma)\cap\mathcal
W(\beta)$ is finite, or $\mathcal W(\gamma)\cap\mathcal W(\gamma)=\mathcal
W(\sigma)$ for some combinatorial geodesic ray $\sigma$, by
Lemma~\ref{lem:raybuild} below.

As discussed in~\cite{Chepoi2000,Roller98}, the 1-skeleton of $\mathbf X$ is a
\emph{median graph}, which means that for any three distinct 0-cubes $x,y,z$,
there exists a unique 0-cube $m=m(x,y,z)$ such that the combinatorial distance
between any two of $x,y,z$ is realized by a geodesic segment in $\mathbf
X^{(1)}$ that passes through $m$.  In terms of hyperplanes, this means that the
set of hyperplanes $H$ separating $x$ from $m$ is exactly the set of $H$ such
that $H$ separates $x$ from $y$ and $H$ separates $x$ from $z$.

\begin{lem}\label{lem:raybuild}
Let $\mathbf X$ be a CAT(0) cube complex and let $\beta,\gamma:[0,\infty)\rightarrow\mathbf X$ be combinatorial geodesic rays, and suppose that $\mathcal W(\gamma)\cap\mathcal W(\beta)$ is infinite.  Then there exists a combinatorial geodesic ray $\sigma$ such that $[\sigma]\leq[\beta]$ and $[\sigma]\leq[\gamma]$.
\end{lem}

\begin{proof}
By Lemma~\ref{lem:folding}, we may assume that $\gamma(0)=\beta(0)$.  For each $t\geq 0$, let $m_t$ be the median of the 0-cubes $\beta(0),\beta(t)$, and $\gamma(t)$.  Let $\sigma_t$ be a combinatorial geodesic segment joining $m_t$ to $\beta(0)$.  By the definition of the median, each hyperplane crossing $\sigma_t$ separates $\beta(t)$ and $\gamma(t)$ from $\beta(0)$.  Hence $\mathcal W(\sigma_t)\subset\mathcal W(\gamma)\cap\mathcal W(\beta)$.  On the other hand, if $W$ crosses both $\gamma$ and $\beta$, then $W$ separates $m_t$ from $\beta(0)$ for all sufficiently large $t$, so that
\[\mathcal W(\gamma)\cap\mathcal W(\beta)=\bigcup_{t\geq 0}\mathcal
W(\sigma_t).\]

Now, $\sigma_0=\beta(0)$, and for $t\geq 1$, choose
$\sigma_t=\sigma_{t-1}\alpha_t$, where $\alpha_t$ is a combinatorial geodesic
segment joining $m_{t-1}$ to $m_t$.  By induction, $\sigma_{t-1}$ is a geodesic
segment, and $\alpha_t$ is a geodesic segment by definition, so either
$\sigma_t$ is a geodesic segment joining $\sigma(0)=\beta(0)$ to $m_t$, or some
hyperplane $H$ is dual to a 1-cube of $\sigma_{t-1}$ and a 1-cube of
$\alpha_t$.  Since $H$ crosses $\sigma_{t-1}$, both $\beta(t-1)$ and
$\gamma(t-1)$ are separated from $\beta(0)$ by $H$, by the definition of the
median $m_{t-1}$. Thus both $\beta(t)$ and $\gamma(t)$ are separated from
$\beta(0)$ by $H$, whence $m_t$ is separated from $\beta(0)$ by $H$.  In other
words, $H$ separates $m_t$ from $\beta(0)$, and $H$ separates $m_{t-1}$ from
$\beta(0)$, and $H$ separates $m_t$ from $m_{t-1}$, since it crosses the
geodesic segment $\alpha_t$.  This is a contradiction, and each $\sigma_t$ is
therefore a geodesic segment.

Now, for all $t$, we have $\sigma_t\subseteq\sigma_{t+1}$.  K\"{o}nig's lemma
now supplies us with a combinatorial ray $\sigma=\cup_{t\geq 0}\sigma_t$ that
is geodesic (since it crosses each hyperplane in at most one 1-cube) and has
the property that $\cup_{t\geq 0}\mathcal W(\sigma_t)=\mathcal W(\sigma)$.  It
was shown above that $\cup_{t\geq 0}\mathcal W(\sigma_t)=\mathcal
W(\beta)\cap\mathcal W(\gamma)$.  By definition, $[\sigma]\leq[\gamma]$ and
$[\sigma]\leq[\beta]$.
\end{proof}

\begin{lem}\label{lem:raybuild2}
Let $\mathbf X$ be a CAT(0) cube complex and let $\beta,\gamma:[0,\infty)\rightarrow\mathbf X$ be combinatorial geodesic rays with $\beta(0)=\gamma(0)$.  Suppose that for all $U\in\mathcal W(\beta)$ and $V\in\mathcal W(\gamma)$, the hyperplanes $U$ and $V$ cross.  Then there exists a combinatorial geodesic ray $\sigma$ such that $[\beta],[\gamma]\leq[\sigma]$.
\end{lem}

\begin{proof}
Let $x=\beta(0)$ and, for each $t\in\naturals$, let $y_t=\beta(t)$ and
$z_t=\gamma(t)$.  Define a 0-cube $m_t$ by orienting all hyperplanes, as
follows.  If $W\in\mathcal W(\beta)$, let $m_t(W)=y_t(W)$ be the halfspace
associated to $W$ that contains $y_t$.  If $W\in\mathcal W(\gamma)$, let
$m_t(W)=z_t(W)$.  Otherwise, if $W$ does not cross $\beta$ or $\gamma$, let
$m_t(W)=x(W)$.  Now, $m_t(W)\neq x(W)$ if and only if $W$ is one of the
finitely many hyperplanes separating $x$ from $y_t$ or $z_t$ and hence the
orientation $m_t$ defines a 0-cube, also denoted $m_t$, provided it orients the hyperplanes consistently, which we now verify.

Let $W,W'$ be hyperplanes.  If neither $W$ nor $W'$ crosses $\beta$ or $\gamma$,
then $m_t(W)\cap m_t(W')=x(W)\cap x(W')\neq\emptyset$, since $x$ is a
0-cube.  If $W$ crosses $\gamma$ and $W'$ crosses neither $\beta$ nor $\gamma$,
then $m_t(W')=x(W')$ and either $W$ and $W'$ cross or $W$ and $x$ lie in the
same halfspace associated to $W'$.  In either case, $m_t(W)\cap
m_t(W')\neq\emptyset$.  Finally, if $W$ crosses $W'$, then $W$ and $W'$ cross,
so that $m_t(W)\cap m_t(W')\neq\emptyset$.  Thus $m_t$ consistently orients all
hyperplanes.  Hence there is a 0-cube $m_t$ such that a hyperplane $W$ separates
$x$ from $m_t$ if and only if $W$ separates $x$ from $y_t$ or from $z_t$.

Let $\sigma_t$ be a combinatorial geodesic segment joining $x$ to $m_t$.  By
construction, $m_t$ is separated from $x$ by exactly the set of hyperplanes
$W$ that separate $x$ from $y_t$ or from $z_t$.  Thus
\[\bigcup_{t\geq 0}\mathcal W(\sigma_t)=\mathcal W(\beta)\cup\mathcal
W(\gamma),\]
as is illustrated heuristically in Figure~\ref{fig:raybuild2}.

As in the proof of Lemma~\ref{lem:raybuild}, we need a particular choice of
$\sigma_t$, made in the following inductive way.  First,
$\sigma_0=\beta(0)=\gamma(0)$.  Next, for $t\geq 1$, let $\alpha_t$ be a
geodesic segment joining $m_{t-1}$ to $m_t$ and let
$\sigma_t=\sigma_{t-1}\alpha_t$.  Suppose the hyperplane $H$ crosses $\alpha_t$
and $\sigma_{t-1}$.  Since it crosses $\sigma_{t-1}$, the hyperplane $H$
separates exactly one of $\beta(t-1)$ or $\gamma(t-1)$ from $\beta(0)$; suppose
the former.  Then $H$ separates $\beta(t)$ from $\beta(0)$, and thus separates
$m_t$ from $\beta(0)$.  Since $H$ crosses $\alpha_t$, the 0-cubes $m_t,m_{t-1}$
are separated by $H$, but both lie in the halfspace associated to $H$ that does
not contain $\beta(0)$, and this is a contradiction.  Hence $\sigma_t$ is a
geodesic joining $\beta(0)$ to $m_t$, and $\sigma_t\subseteq\sigma_{t+1}$ for
all $t$.  Arguing as in the proof of Lemma~\ref{lem:raybuild} now yields the
desired $\sigma$.
\end{proof}

\begin{rem}[Isometrically embedded quadrants]\label{rem:quadrant_embed}
Under the hypotheses of Lemma~\ref{lem:raybuild2}, the proof of Theorem~3.23 of~\cite{HagenBoundary} actually yields a copy of $\mathbf R_2$ in $\mathbf X$ that contains geodesic rays almost-equivalent to $\beta,\gamma,$ and has isometrically embedded 1-skeleton; this implies Lemma~\ref{lem:raybuild2}.  Since we do not require this stronger fact, we have opted for a self-contained and slightly simpler proof.
\end{rem}

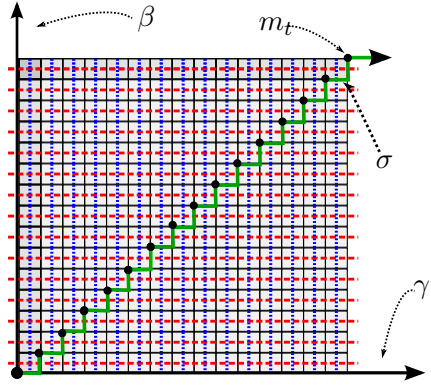
\begin{figure}
\begin{overpic}[width=0.35\textwidth]{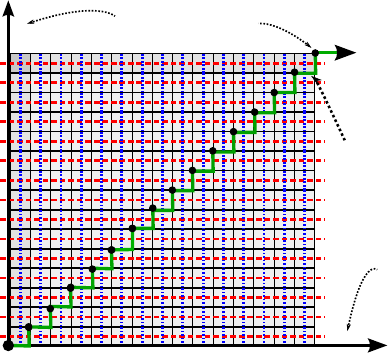}
\put(31,85){$\beta$}
\put(60,84){$m_t$}
\put(88,51){$\sigma$}
\put(97,21){$\gamma$}
\end{overpic}
\caption{The proof of Lemma~\ref{lem:raybuild2}.  The bold 0-cubes are the various $m_t$, while the hyperplanes are dashed line segments.}
\label{fig:raybuild2}
\end{figure}

To define the simplicial boundary requires the following lemmas.

\begin{lem}\label{lem:minimalexists}
Suppose that $\mathbf X$ is locally finite and contains no infinite set of pairwise-crossing hyperplanes, and let $\gamma\subset\mathbf X$ be a combinatorial geodesic ray.  Then there exists a combinatorial geodesic ray $\gamma_0$ such that $[\gamma_0]$ is minimal and $[\gamma_0]\leq[\gamma]$.
\end{lem}

\begin{proof}
Since $\mathcal W(\gamma)$ is infinite, and there is no infinite family of pairwise-crossing hyperplanes in $\mathbf X$, there exists a collection $\{H_n\}_{n\geq 0}\subset\mathcal W(\gamma)$ of hyperplanes such that for all $m,n\geq 0$, the hyperplanes $H_m$ and $H_n$ do not cross.  These may be labeled so that for all $n$, $\gamma(0)$ is separated from $H_{n+1}$ by $H_n$.  For each $n\geq 0$, let $\{\sigma_n^i\}_{i\in I_n}$ be the set of combinatorial geodesic segments joining $\gamma(0)$ to a closest point of $N(H_n)$.  Then $\mathcal W(\sigma_n^i)$ is the set of hyperplanes that separate $\gamma(0)$ from $H_n$, and since $\gamma$ contains $\gamma(0)$ and intersects $H_n$, each of these hyperplanes must cross $\gamma$, i.e. $\mathcal W(\sigma_n^i)\subset\mathcal W(\gamma)$.

For $n\geq 0$, let $\mathcal S_n$ be the set of combinatorial geodesic segments $\sigma$, with $\sigma(0)=\gamma(0)$, that terminate on $N(H_n)$ and extend to a path of the form $\sigma^j_m$ for some $m\geq n$.  In particular, by considering $m=n$, we see that each $\sigma^j_n\in\mathcal S_n$.  Since $\mathbf X$ is locally finite, each element of $\mathcal S_n$ can be extended in at most finitely many ways to an element of $\mathcal S_{n+1}$, and each element of $\mathcal S_n$ contains a unique element of $\mathcal S_{n-1}$.  Hence K\"{o}nig's lemma yields a sequence $\sigma_0\subset\sigma_1\subset\ldots$ of combinatorial geodesic segments whose union $\gamma_0$ is a combinatorial geodesic ray with $\gamma_0(0)=\gamma(0)$.  Each hyperplane $H$ crossing $\gamma_0$ crosses a path of the form $\sigma^j_n$ for some $n\geq 0$, and thus belongs to $\mathcal W(\gamma)$.  Hence $[\gamma_0]\leq[\gamma]$, and each $H_n$ crosses $\gamma_0$.

Now, for all $n$, if the hyperplane $H$ crosses $\sigma_n$, then $H$ must separate $H_n$ from $\gamma(0)$.  Thus $\mathcal W(\gamma_0)$ consists of the set $\{H_n\}$, together with those hyperplanes $H$ that separate some $H_n$ from $\gamma(0)$.  Let $\beta$ be a combinatorial geodesic ray with $[\beta]\leq[\gamma_0]$.  Then each hyperplane crossing $\beta$ is either one of the $H_n$, or separates some $H_n$ from $\gamma(0)$, or belongs to the finite set $\mathcal W(\beta)-\mathcal W(\beta)\cap\mathcal W(\gamma_0)$.  Since every hyperplane crossing $\gamma_0$ is of one of the former two types, it follows that $[\beta]=[\gamma_0]$, so that $[\gamma_0]$ is minimal.
\end{proof}

\begin{lem}\label{lem:dimdefn}
Suppose that $\mathbf X$ is locally finite and contains no infinite set of pairwise-crossing hyperplanes, and let $\gamma\subseteq\mathbf X$ be a combinatorial geodesic ray.  Then there exists an integer $D\geq 0$ and combinatorial geodesic rays $\gamma_0,\ldots,\gamma_D$ such that for $0\leq i\leq d$, the class $[\gamma_i]$ is minimal, $[\gamma_i]\leq[\gamma]$, and if $[\sigma]\leq\gamma$ is minimal, then $[\sigma]=[\gamma_i]$ for some $i$.
\end{lem}

\begin{proof}
By Lemma~\ref{lem:minimalexists} and Lemma~\ref{lem:folding}, the set $\mathcal M$ of geodesic rays $\beta$ such that $\beta(0)=\gamma(0)$, $[\beta]\leq[\gamma]$, and $[\beta]$ is minimal, is nonempty.  Now, for any $\beta\in\mathcal M$, let $H$ be a hyperplane crossing $\gamma$ but not crossing $\beta$.  Then for all but finitely many hyperplanes $V$ crossing $\beta$, $H$ and $V$ must cross.  Indeed, if $V$ crosses $\beta$ and $\gamma$, and is dual to a 1-cube of $\gamma$ that is separated in $\gamma$ from $\gamma(0)$ by the 1-cube dual to $H$, then there is a geodesic triangle $ABC$, where $B$ is a path in $N(V)$ starting on $\beta$ and ending on $\gamma$, and $A,C$ are the subpaths of $\beta,\gamma$ between $\gamma(0)$ and $N(V)$.  This triangle bounds a disc diagram, one of whose dual curves emanates from $C$ and maps to $H$.  This dual curve must end on $B$ because $H$ does not cross $\beta$, and hence $H$ crosses $V$.

Therefore, if $\beta,\beta'\in\mathcal M$ and $[\beta]\neq[\beta']$, then for all but finitely many $H\in\mathcal W(\beta),H'\in\mathcal W(\beta')$, the hyperplanes $H,H'$ cross.  In fact, the above argument, together with the proof of~\cite[Lemma~6.4]{HagenBoundary} shows $\beta,\beta'$ can be chosen within their equivalence classes in such a way that $H,H'$ cross whenever $H$ crosses $\beta$ but not $\beta'$ and $H$ crosses $\beta'$ but not $\beta$.

Suppose that the set $[\mathcal M]$ of equivalence classes of rays in $\mathcal M$ is infinite.  Suppose also that for some $n\geq 1$, we have $\beta_1,\ldots,\beta_n\in\mathcal M$ with the following properties:
\begin{enumerate}
 \item $[\beta_i]\neq[\beta_j]$ for $i\neq j$.
 \item If $H$ crosses $\beta_i$ and $H'$ crosses $\beta_j$, with $i\neq j$, then either $H$ and $H'$ cross, or at least one of $H,H'$ crosses both $\beta_i$ and $\beta_j$.
 \item For $i\neq j$, the set $\mathcal W(\beta_i)\cap\mathcal W(\beta_j)$ is equal to the set of hyperplanes that cross a geodesic path $P_{i,j}=\beta_i\cap\beta_j$ whose initial point is $\gamma(0)$.
 \item $P_{i-1,i}\subseteq P_{i,i+1}$ and $P_{i,j}\subseteq P_{i,j'}$ for $2\leq i<j'\leq j\leq n$.
\end{enumerate}
Some such collections of rays are illustrated in Figure~\ref{fig:raystree}.  Note that the above conditions are vacuously satisfied when $n=1$.

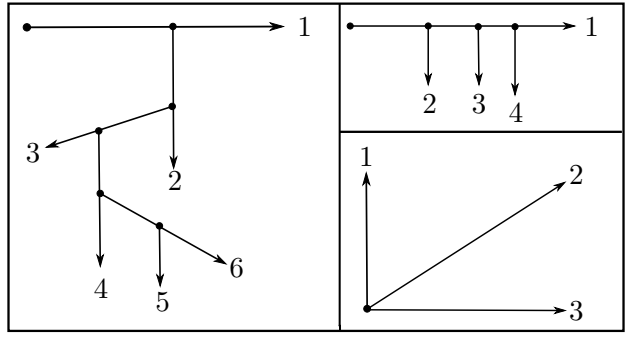
\begin{figure}
\begin{overpic}[scale=0.40]{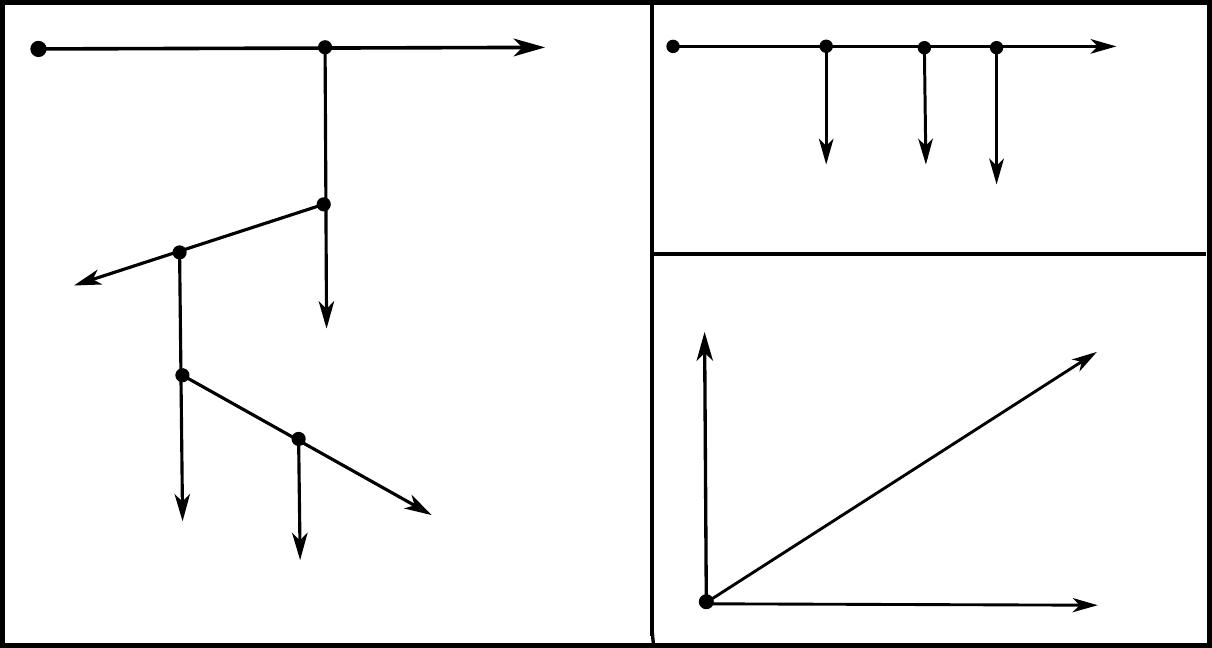}
\put(47,48){$1$}
\put(26,23){$2$}
\put(3,27.5){$3$}
\put(14,5.5){$4$}
\put(24,3.5){$5$}
\put(36,9){$6$}
\put(93.5,48){$1$}
\put(67.25,35.5){$2$}
\put(75.25,35.5){$3$}
\put(81.25,34){$4$}
\put(57,27){$1$}
\put(91,24){$2$}
\put(91,2){$3$}
\end{overpic}
\caption{The ``generic'' case is shown at left.  At right are two interesting special cases.}
\label{fig:raystree}
\end{figure}

Now, since $\mathcal M$ contains infinitely many equivalence classes of rays, there exists $\beta_{n+1}'$ that is inequivalent to $\beta_i$ for $1\leq i\leq n$.  The folding argument used in the proof of~\cite[Lemma~3.22]{HagenBoundary} shows that the set $\{[\beta_1],\ldots,[\beta_n],[\beta'_{n+1}]\}$ is represented by a set of $n+1$ combinatorial geodesic rays, each crossing the exact same set of hyperplanes as some $\beta_i$ or $\beta_{n+1}'$, satisfying the properties listed above.  We thus have a set $\{[\beta_n]\}_{n\in\naturals}$ of distinct equivalence classes of rays such that $\mathcal W_{n-1}\cap\mathcal W_n\subseteq\mathcal W_{n}\cap\mathcal W_{n+1}$ for $n\geq 2$, where $\mathcal W_n=\mathcal W(\beta_n)$, and $\mathcal W_{n}\cap\mathcal W_{m'}\subseteq\mathcal W_n\cap\mathcal W_m$ when $n<m\leq m'$.  Note that by minimality of the $\beta_i$, each of the preceding intersections is finite.

For each $n\geq 1$, let $H_n\in\mathcal W_n$ be a hyperplane that does not belong to $\mathcal W_{n+1}$.  Such an $H_n$ exists since $|\mathcal W_m\cap\mathcal W_n|<\infty$ for $m\neq n$.  For any $m>n$, $H_n$ does not lie in $\mathcal W_m$, since $\mathcal W_m\cap\mathcal W_n\subseteq\mathcal W_{n+1}\cap\mathcal W_n$.  On the other hand, $H_n\not\in\mathcal W_k$ for $k<m$, since $\mathcal W_n\cap\mathcal W_{n-1}\subseteq\mathcal W_n\cap\mathcal W_{n+1}$.  Thus $H_m\neq H_n$ for $m\neq n$, and moreover these two hyperplanes cross, since two hyperplanes in $\mathcal W_m\cup\mathcal W_n$ cross unless both belong to the intersection of those sets.  It follows that $\mathbf X$ contains an infinite set of pairwise-crossing hyperplanes, a contradiction.  Thus $|[\mathcal M]|=D+1$ for some integer $D\geq 0$.
\end{proof}

\begin{rem}
Note that we have required that any collection of pairwise crossing hyperplanes in $\mathbf X$ is finite, but we have not required a uniform upper bound on the cardinality of such collections, i.e. we do not need $\mathbf X$ to have finite dimension.  Throughout this paper, we always impose the former requirement on sets of pairwise-crossing hyperplanes.  Finite-dimensionality comes into play in the next section, where there is a cocompact group action.
\end{rem}

\begin{defn}[Invisible set]\label{defn:invisibleset}
Let $\gamma$ be a combinatorial geodesic ray and let $\gamma_0,\ldots,\gamma_D$ be geodesic rays provided by Lemma~\ref{lem:dimdefn}.  If $\mathcal U(\gamma)=\mathcal W(\gamma)-\cup_{i=0}^D\mathcal W(\gamma_i)$ is infinite, we call $\mathcal U(\gamma)$ an \emph{invisible set} for $\gamma$.  If $\mathcal U(\gamma)$ and $\mathcal U(\gamma')$ are invisible sets for the rays $\gamma,\gamma'$, then they are \emph{equivalent} if their symmetric difference is finite.  One can show that if $\mathcal U(\gamma)$ is an invisible set for a ray $\gamma'$, then $[\gamma']=[\gamma]$, but we shall not use this fact here.
\end{defn}

The \emph{simplicial boundary} $\simp\mathbf X$ is defined as follows.  First, for each minimal $[\gamma]\in\mathfrak R\mathbf X$, there is a 0-simplex in $\simp\mathbf X$.  If $\gamma_0,\ldots,\gamma_D$ are rays representing equivalence classes corresponding to 0-simplices, then these 0-simplices span a $D$-simplex if there is a ray $\gamma$ such that $[\gamma_i]\leq[\gamma]$ for $0\leq i\leq D$.  If $\mathcal W(\gamma)$ contains an invisible set, then we add a $(D+1)$-simplex that is the join of the above $D$-simplex with the 0-simplex corresponding to the invisible set $\mathcal U(\gamma)$.  Hence each equivalence class of rays in $\mathbf X$ determines a finite-dimensional simplex of $\simp\mathbf X$, by Lemma~\ref{lem:dimdefn}.  Moreover, since each invisible set is a subset of $\mathcal W(\gamma)$ for some ray $\gamma$, each maximal simplex of $\simp\mathbf X$ is determined by an equivalence class of rays, and is therefore finite-dimensional.

\begin{rem}[Invisible simplices]
By definition, $\simp\mathbf X$ may contain simplices that do not arise as almost-equivalence classes of combinatorial geodesic rays.  For example, if $\mathbf X$ is the cubical sector shown in Figure~\ref{fig:eighthflat}, then $\simp\mathbf X$ is a 1-simplex corresponding to the almost-equivalence class of diagonal geodesic rays that cross all but finitely many of the hyperplanes.  One of the 0-simplices is represented by the horizontal geodesic ray whose set of dual hyperplanes is exactly the set of vertical hyperplanes in $\mathbf X$, but since no ray crosses only horizontal hyperplanes, the other 0-simplex does not correspond to a class of rays.  Crucially, by Theorem~3.19 of~\cite{HagenBoundary}, maximal simplices of the simplicial boundary are visible, and the proof of that theorem also ensures that each invisible simplex is contained in a unique maximal simplex.  

\begin{figure}
  \vspace{-10pt}
  \includegraphics[width=0.25\textwidth]{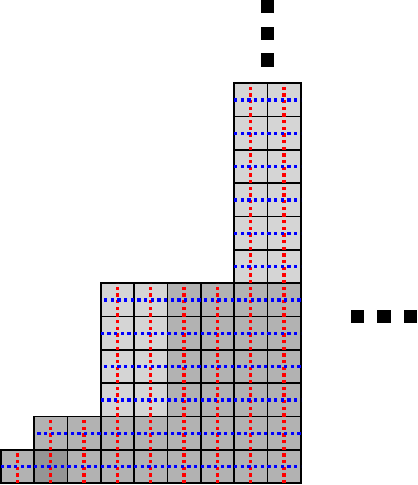}\\
  \caption{The simplicial boundary of an eighth-flat is a 1-simplex, one of whose 0-simplices is ``invisible'' in the cube complex.  Hyperplanes are shown as dotted line segments.}\label{fig:eighthflat}
  \vspace{-10pt}
\end{figure}
\end{rem}

\begin{exmp}\label{exmp:simplicialhyperbolic}
If $\mathbf X$ is an infinite, locally finite tree, then any two geodesic rays are either almost-equivalent, or neither consumes the other.  Hence $\simp\mathbf X$ is a discrete set of 0-simplices.  More generally, it is shown in~\cite{HagenBoundary} that if $\mathbf X$ is infinite and hyperbolic, then $\simp\mathbf X$ is a discrete set of 0-simplices.  If $\mathbf X$ is the standard tiling of $[0,\infty)^2$ by 2-cubes, then $\simp\mathbf X$ is a 1-simplex.  If $\mathbf X$ is the standard tiling of $[0,\infty)\times\reals$ by 2-cubes, then $\simp\mathbf X$ is a subdivided interval of length 2.  If $\mathbf X$ is the standard tiling of $\reals^2$ by 2-cubes, then $\simp\mathbf X$ is a 4-cycle.
\end{exmp}

Lemma~\ref{lem:highdimflat} describes the simplicial boundary of the standard tiling of Euclidean space by cubes, and follows from Theorem~3.28 of~\cite{HagenBoundary}.

\begin{lem}\label{lem:highdimflat}
Let $n\geq 1$ and let $\mathbf R_n$ be the standard tiling of $\Euclidean^n$ by $n$-cubes.  Then $\simp\mathbf R_n$ is isomorphic to the $(n-1)$-dimensional hyperoctahedron $\mathbf Q_n$.
\end{lem}

\begin{rem}
By regarding $\simp\mathbf X$ as being constructed from right-angled spherical simplices in which 1-simplices have length $\frac{\pi}{2}$, one realizes $\simp\mathbf X$ as a CAT(1) space.  It is shown in~\cite[Proposition~3.37]{HagenBoundary} that if $\mathbf X$ is fully visible, then $\simp\mathbf X$, as a CAT(1) space, isometrically embeds in the Tits boundary of $\mathbf X$ (when the latter is endowed with the piecewise-Euclidean CAT(0) metric in which all cubes are Euclidean unit cubes).
\end{rem}


Part of the utility of the simplicial boundary comes from the fact that an action of a group $G$ on $\mathbf X$ induces an action of $G$ on $\simp\mathbf X$.

\begin{prop}\label{prop:inducedaction}
Suppose $G$ acts on the CAT(0) cube complex $\mathbf X$.  Then $G$ acts by simplicial automorphisms on $\simp\mathbf X$.
\end{prop}

\begin{proof}
Let $g\in G$ and let $\gamma:[0,\infty)\rightarrow\mathbf X$ be a combinatorial geodesic ray.  Then $g\gamma$ is also a combinatorial geodesic ray, since $G$ acts by isometries.  Now, if $\gamma'$ is a geodesic ray for which $|\mathcal W(\gamma)\cap\mathcal W(\gamma')|<\infty$, then since $G$ acts on the set of hyperplanes, the set
\[\mathcal W(g\gamma)\cap\mathcal W(g\gamma')=g\mathcal W(\gamma)\cap g\mathcal W(\gamma')=g(\mathcal W(\gamma)\cap\mathcal W(\gamma'))\]
is finite, being a translate of a finite set.  Hence $[g\gamma]=[g\gamma']$, i.e. the $G$-action on the set of combinatorial geodesic rays preserves almost-equivalence classes, and so we define a $G$-action on $\mathfrak R\mathbf X$ by $g[\gamma]=[g\gamma]$.  Analogous considerations show that, for all $g\in G$ and all $[\gamma],[\gamma']\in\mathfrak R\mathbf X$, if $[\gamma]\leq[\gamma']$, then $g[\gamma]\leq g[\gamma']$, and thus $G$ preserves the partial order $\leq$.  Hence $G$ acts by simplicial automorphisms on $\simp\mathbf X$.
\end{proof}

One can check that elliptic elements of $G$ may act trivially or nontrivially on $\simp\mathbf X$, but that if $g\in G$ acts as a hyperbolic isometry of $\mathbf X$, then there are two distinct simplices of $\simp\mathbf X$, corresponding to the ends of an axis for $g$, each of which is stabilized by $g$.

\subsubsection{Boundaries of cubulations of virtually-$\integers$ groups}\label{subsec:virtually_Z_cubulations}
The following fact forms the base case in our inductive proof of Theorem~\ref{thm:boundaryoctahedron} by describing the simplicial boundary of a cocompact cubulation of a virtually $\integers$ group.

\begin{lem}\label{lem:case1}
Let the group $G$ act properly and cocompactly on a CAT(0) cube complex $\mathbf X$ that is quasi-isometric to $\reals$.  Then $\simp\mathbf X$ consists of two non-adjacent 0-simplices.
\end{lem}

An example of a cube complex of the type described in Lemma~\ref{lem:case1} is shown in Figure~\ref{fig:zcocompact}.

\begin{proof}[Proof of Lemma~\ref{lem:case1}]
$\mathbf X$ is quasi-isometric to $\reals$ and to $G$, and thus $G$ is 2-ended.  Hence there is a finite-index infinite cyclic subgroup $G'\leq G$, generated by an element $b$, and $G'$ acts properly and cocompactly on $\mathbf X$.  The element $b$ cannot be elliptic, and, by passing if necessary to a further finite-index cyclic subgroup, we may assume that $b$ is combinatorially hyperbolic.  Let $\beta:\reals\rightarrow\mathbf X$ be a combinatorial geodesic axis for $b$, and write
\[\beta_+=\beta([0,\infty)),\,\,\beta_-=\beta((-\infty,0]).\]
Let $v_+$ be the simplex of $\simp\mathbf X$ represented by $\beta_+$ and let $v_-$ be the simplex represented by $\beta_-$.  Now, since $\beta$ is a geodesic, no hyperplane crosses $\beta_+$ and $\beta_-$, so that $v_+$ and $v_-$ are distinct simplices.

Let $\gamma:[0,\infty)\rightarrow\mathbf X$ be a combinatorial geodesic ray.  Let $\kappa\geq 0$ be the quasi-surjectivity constant of the map $\beta:\reals\rightarrow\mathbf X$.  For all $t\geq 0$,
\[d_{\mathbf X}(\gamma(t),\beta(t))\leq \kappa.\]
Hence $\gamma$ lies in the $\kappa$-neighborhood of $\beta_+$ or $\beta_-$.  Without loss of generality, we can assume that $d_{\mathbf X}(\gamma(t),\beta_+)\leq \kappa$ for all $t\geq 0$.

Let $U\in\mathcal W(\gamma)-\mathcal W(\beta_+)$ be dual to the 1-cube $\gamma([s,s+1])$.  Then $U$ separates $\gamma([s+1,\infty))$ from $\beta_+$, or $U$ separates $\gamma(0)$ from $\beta_+(0)$.  Hence $|\mathcal W(\gamma)-\mathcal W(\beta_+)|\leq \kappa+d_{\mathbf X}(\gamma(0),\beta_+(0)).$
On the other hand, suppose that $V\in\mathcal W(\beta_+)-\mathcal W(\gamma)$.  Then either $V$ separates $\gamma(0)$ from $\beta_+(0)$, or $V$ separates $\gamma$ from an infinite sub-ray of $\beta_+$, and we conclude that $|\mathcal W(\gamma)\triangle\mathcal W(\beta_+)|\leq 2\kappa+2d_{\mathbf X}(\gamma(0),\beta_+(0)),$
so that $[\gamma]=[\beta_+]$.  Hence every combinatorial geodesic ray is almost-equivalent to either $\beta_+$ or $\beta_-$, and thus $v_+$ and $v_-$ are distinct 0-simplices whose disjoint union is the whole of $\simp\mathbf X$.
\end{proof}

\begin{figure}
  \includegraphics[width=0.5\textwidth]{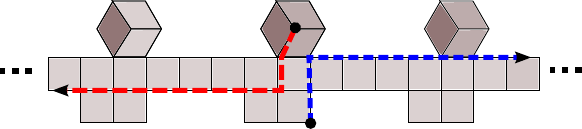}\\
  \caption{Part of a $\integers$-cocompact CAT(0) cube complex.  Each combinatorial geodesic ray is almost-equivalent to one of the two dashed rays.}\label{fig:zcocompact}
  \vspace{-10pt}
\end{figure}

\section{Hyperoctahedral boundary}\label{sec:bounarychar}
In this section, which is devoted to the proof of Theorem~A, $G$ is a group that acts properly and cocompactly on a CAT(0) cube complex $\mathbf X$ and contains a finite-index subgroup $T_{_G}\cong\integers^n$ for some $n\geq 1$.  Let $H$ be a $G$-essential hyperplane, and let $\mathbf H$ be the carrier of $H$. Let $G_{n-1}=\stabilizer_{\mathbf X}(H)$, and let $T_{n-1}=G_{n-1}\cap T_{_G}$.  We have in mind the case where $G$ is an $n$-dimensional crystallographic group and $T_{_G}$ is the translation subgroup.  In fact, it will be convenient to reduce to this case.

\begin{thm}\label{thm:boundaryoctahedron}
Let $G$ be a group with a finite-index subgroup $T_{_G}\cong\integers^n$, with $n\geq 1$.  If $G$ acts properly and cocompactly on the CAT(0) cube complex $\mathbf X$, then the simplicial boundary of $\mathbf X$ is isomorphic to $\mathbf Q_n$, the $(n-1)$-dimensional hyperoctahedron.
\end{thm}

\begin{proof}
The claim follows from Lemma~\ref{lem:case1} when $n=1$.  For $n\geq 2$, we first note that we can assume that $G$ is an $n$-dimensional crystallographic group; this is needed in the proof of Lemma~\ref{lem:induction1}, where we require a $G$-equivariant quasi-isometry $\mathbf X\rightarrow\Euclidean^n$.  The assumption that $G$ is crystallographic is justified by noting that $G$ has a finite-index subgroup $G'$ with a geometric action on $\Euclidean^n$ -- at minimum, $T_{_G}$ is such a subgroup.  The subgroup $G'$ acts properly and cocompactly on $\mathbf X$ if $G$ does, and the desired conclusion is a statement about $\mathbf X$.  We thus henceforth assume that $G$ is an $n$-dimensional crystallographic group whose translation subgroup $T_{_G}$ is the kernel of the epimorphism $\psi:G\rightarrow P_{_G}$, where $P_{_G}$ is the point group.

We have $\simp\mathbf H\cong\mathbf Q_{n-1}$, by induction, since, by Lemma~\ref{lem:bnminus1}, $\mathbf H$ is a CAT(0) cube complex on which $T_{n-1}$ acts properly and cocompactly.  By Lemma~\ref{lem:induction1}, $\simp\mathbf X\cong\simp\mathbf H\star\mathbf Q_1\cong\mathbf Q_n$.
\end{proof}


\begin{rem}\label{rem:alternative}
There is an alternative proof of Theorem~\ref{thm:boundaryoctahedron} using the \emph{rank-rigidity theorem} of~\cite{CapraceSageev}, along with results in~\cite{HagenBoundary}.  Indeed, if $\integers^n$ acts properly and cocompactly on $\mathbf X$, then since $\integers^n$ contains no rank-one element for $n\geq 1$, the $\integers^n$-\emph{essential core} of $\mathbf X$ is a product $\mathbf X_1\times\mathbf X_{n-1}$ of CAT(0) cube complexes, whose simplicial boundary is $\simp\mathbf X_1\star\simp\mathbf X_{n-1}$, by Theorem~3.28 of~\cite{HagenBoundary}.  Lemma~\ref{lem:case1} shows that, if $n=1$, then the essential core has simplicial boundary $\mathbf Q_1$, and it follows by induction that the essential core of $\mathbf X$ has simplicial boundary $\mathbf Q_n$.  The final step is to verify that passing to the essential core does not affect the simplicial boundary.  The proof we give below is self-contained, however, and seems more readily adaptable to other classes of groups and cube complexes.
\end{rem}

\subsection{The inductive step}
Let $n\geq 2$ and let the $n$-dimensional crystallographic group $G$ act properly and cocompactly on $\mathbf X$.  The first step is to show that the CAT(0) cube complex $\mathbf H$ admits a proper, cocompact action by an $(n-1)$-dimensional crystallographic group:

\begin{lem}\label{lem:bnminus1}
$T_{n-1}$ is a finite-index subgroup of $G_{n-1}$, and $T_{n-1}\cong\integers^{n-1}$.
Moreover, $T_{n-1}$ acts cocompactly on $\mathbf H$.
\end{lem}

\begin{proof}
By definition, the kernel of $\psi|_{G_{n-1}}$ is $T_{n-1}$, and hence $T_{n-1}$ has finite index in $G_{n-1}$.  This proves the first assertion.

$T_{_G}$ acts properly and cocompactly on $\mathbf X$, since $T_{_G}$ is a finite-index subgroup of $G$.  Since $H$ is $G$-essential and $T_{_G}$ has finite index, $H$ is $T_{_G}$-essential.  Now, with respect to the $T_{_G}$-action on $\mathbf X$, the stabilizer of $H$ is exactly $T_{n-1}$.  Indeed, $T_{n-1}$ stabilizes $H$ by definition; on the other hand, if $t\in T_{_G}-T_{n-1}$, then $t\not\in G_{n-1}=\stabilizer_{\mathbf X}(H)$.

Since $T_{_G}\cong\integers^n$, we have $T_{n-1}\cong\integers^k$ for some $k\leq n$.  Either $T_{n-1}$ is a codimension-1 subgroup or $T_{n-1}$ has finite index in $T_{_G}$.  Thus $k\geq n-1$.  Suppose that $T_{n-1}$ has finite index in $T_{_G}$.  Then there exists $N\in\integers$ such that for all $t\in T_{_G}$, we have $t^N\in T_{n-1}$, i.e. $T^N(\mathbf H)=\mathbf H$.  This implies that $\mathbf X$ lies in a uniform neighborhood of $\mathbf H$, contradicting $G$-essentiality of $\mathbf H$.  Hence $T_{n-1}$ has infinite index in $T_{_G}$, whence $k=n-1$.  This proves the second assertion.  Lastly, $\mathbf H$ is a convex subcomplex of $\mathbf X$, and thus $G_{n-1}$ acts on $\mathbf H$ cocompactly.  Since $T_{n-1}\leq_{f.i.}G_{n-1}$, the induced action of $T_{n-1}$ on $\mathbf H$ is cocompact.
\end{proof}

\begin{lem}\label{lem:induction1}
$\simp\mathbf X\cong\simp\mathbf H\star\mathbf Q_1$.
\end{lem}

\begin{proof}
The proof has two parts: we first decompose $\simp\mathbf X$ along a subcomplex isomorphic to $\simp H$, and then show that each of the pieces is isomorphic to the join of $\simp H$ with a single 0-simplex.  For use in the latter part of the proof, denote by $\eta:\mathbf X\rightarrow\Euclidean^n$ a $G$-equivariant $(\lambda,\mu)$-quasi-isometry that is $\kappa$-quasi-surjective, for some $\lambda\geq 1,\,\mu,\kappa\geq 0$.

\textbf{Decomposing $\simp\mathbf X$ along $\simp\mathbf H$:}  Since $\mathbf H$ is a convex subcomplex of $\mathbf X$, there is a subcomplex $A\subset\simp\mathbf X$, isomorphic to $\simp\mathbf H$, that consists of those simplices represented by rays in $\mathbf H$ (see Theorem~3.14 of~\cite{HagenBoundary}).  Moreover, since $\mathbf C=\mathfrak h(H)\cup\mathbf H$ and $\mathfrak h^*(H)\cup\mathbf H=\mathbf C^*$ are convex subcomplexes of $\mathbf X$, the same theorem implies that there are subcomplexes $E,E^*$ of $\simp\mathbf X$ such that $E\cong\simp\mathbf C$ consists of simplices with representative rays in $\mathbf C$ and $E^*\cong\simp\mathbf C^*$ consists of simplices with representative rays in $\mathbf C^*$.  (We can always define $A$ to be the subcomplex consisting of all simplices represented by rays in $\mathbf H$, but we need convexity to ensure that $\mathbf H$ has a well-defined simplicial boundary.)

Now, $\simp\mathbf X=E\cup E^*$ and $E\cap E^*=A$.  Indeed, let $\gamma:[0,\infty)\rightarrow\mathbf X$ be a combinatorial geodesic ray.  Since $\gamma$ contains at most one 1-cube dual to $H$, all but finitely many 0-cubes of $\gamma$ lie in $\mathbf C$, or all but finitely many 0-cubes of $\gamma$ lie in $\mathbf C^*$, since $\gamma$ contains a sub-ray that lies entirely in one of the halfspaces associated to $H$.  Hence the simplex of $\simp\mathbf X$ represented by $\gamma$ belongs to $E$ or $E^*$, whence $\simp\mathbf X=E\cup E^*$.

On the other hand, suppose that $\gamma$ represents a simplex $v$ of $E\cap E^*$.  By definition, there exist combinatorial geodesic rays $\alpha$ and $\alpha^*$ such that $[\alpha]=[\alpha^*]$ represents $v$ and $\alpha\subset\mathbf C$ and $\alpha^*\subset\mathbf C^*$.  Without loss of generality, $\mathcal W(\alpha)=\mathcal W(\alpha^*)$ and $H$ separates $\alpha$ from $\alpha^*$.  Let $U$ be a hyperplane that crosses $\alpha$.  Since $H$ separates $\alpha$ from $\alpha^*$ and $U$ is dual to a 1-cube of $\alpha^*$, it follows that $U$ crosses $H$.  Let $\mathcal V$ be the finite set of hyperplanes that separate $\alpha(0)$ from $\mathbf H$.  If $V\in\mathcal V$, then $V$ cannot be dual to a 1-cube of $\alpha$, since $V$ does not cross $H$, and hence $V$ separates $\alpha$ from $\mathbf H$.  This forces $V$ to cross each $U\in\mathcal W(\alpha)$. For each integer $t\geq 0$, define an orientation $x_t$ of the hyperplanes of $\mathbf X$ as follows.  First, if $W$ is a hyperplane that does not belong to $\mathcal W(\alpha)$ or to $\mathcal V$, let $x_t(W)=\alpha(0)(W)$ be the halfspace containing $\alpha(0)$.  If $W\in\mathcal W(\alpha)$, let $x_t(W)=\alpha(t)(W)$ be the halfspace containing the 0-cube $\alpha(t)$.  If $W\in\mathcal V$, let $x_t(W)$ be the halfspace containing $\mathbf H$, i.e. the halfspace \emph{not} containing $\alpha(t)$.  Equivalently, $x_t$ orients each $W\in\mathcal V$ toward the unique 0-cube of $\mathbf H$ that is closest to $\alpha(0)$ (this 0-cube exists because convex subgraphs of median graphs are \emph{gated}~\cite{Chepoi2000}).

Since the consistent orientation $W\mapsto\alpha(t)(W)$ of all hyperplanes differs from $x_t$ only on $\mathcal V$, it suffices to check that $x_t(W)\cap x_t(W')\neq\emptyset$ whenever $W\in\mathcal V$.  This is guaranteed when $W'\in\mathcal W(\alpha)$, since $W$ and $W'$ cross in that case.  If $W'\in\mathcal V$, then $x_t$ orients $W$ and $W'$ toward the gate $x_0$ of $\alpha(0)$ in $\mathbf H$, and thus $x_t(W)\cap x_t(W')\neq\emptyset$.  Finally, if $W'$ does not belong to $\mathcal V$, then $\alpha(0)$ lies in the same halfspace associated to $W'$ as does $x_0$, so that $x_t(W')=x_0(W')$.  Either $W$ and $W'$ cross, or $W\subset x_0(W')$, since $x_0(W')$ contains any geodesic segment joining $x_0$ to $\alpha(0)$, by convexity of halfspaces.  Thus $x_t$ orients $W'$ toward $W$, so that $x_t(W)\cap x_t(W')\neq\emptyset$.  Thus $x_t$ is a consistent orientation of all hyperplanes.  Furthermore, $x_t$ differs from $x_0$, and hence from $\alpha(0)$, on finitely many hyperplanes, and thus $x_t$, being
consistent and canonical, is a 0-cube of $\mathbf X$.  By construction, $x_t\in\mathbf H$ for all $t$.

For all $t\geq 0$, the 0-cubes $x_t$ and $x_{t+1}$ are adjacent, since the corresponding orientations differ only on the hyperplane $W_t$ separating $\alpha(t)$ from $\alpha(t+1)$.  Moreover, if the hyperplane $W$ separates $x_s$ from $x_{s+1}$ and $x_t$ from $x_{t+1}$ for some $s,t\geq 1$, then $s=t$, since $\alpha$ is a geodesic ray.  Hence the map $\beta(t)=x_t$ determines a combinatorial geodesic ray $\beta:[0,\infty)\rightarrow\mathbf H$ such that $\mathcal W(\beta)=\mathcal W(\alpha)$.  Hence $[\beta]=[\alpha]=[\gamma]$, whence $v\subset A$.  This proves that $E\cap E^*\subseteq A$, and also establishes the following fact, which we shall use later in the proof:
\begin{center}
\emph{Let $\gamma:[0,\infty)\rightarrow\mathbf X$ be a combinatorial geodesic ray, and suppose there exists $R<\infty$ such that $d_{\mathbf X}(\gamma(t),\mathbf H)\leq R$ for all $t\geq 0$.  Then the simplex of $\simp\mathbf X$ represented by $\gamma$ lies in $A$.}
\end{center}
Indeed, if $\gamma$ lies in the $R$-neighborhood of $\mathbf H$, then all but finitely many hyperplanes crossing $\gamma$ also cross $H$, and we may argue as before to produce a ray in $\mathbf H$ that is almost-equivalent to $\gamma$: by removing a finite initial segment from $\gamma$ if necessary, we can assume that $\mathcal W(\gamma)\subseteq\mathcal W(\mathbf H)$ and thus that every hyperplane that separates some $\gamma(t)$ from $\mathbf H$ also separates $\gamma(0)$ from $\mathbf H$ and hence crosses each hyperplane in $\mathcal W(\gamma)$.  See Figure~\ref{fig:fellowtravel}.
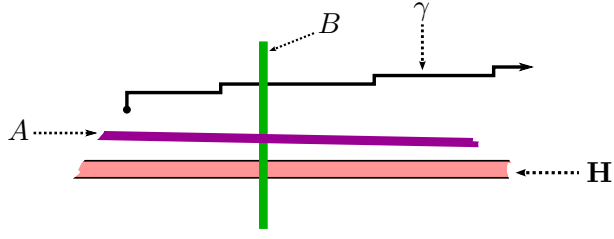
\begin{figure}
\begin{overpic}[width=0.5\textwidth]{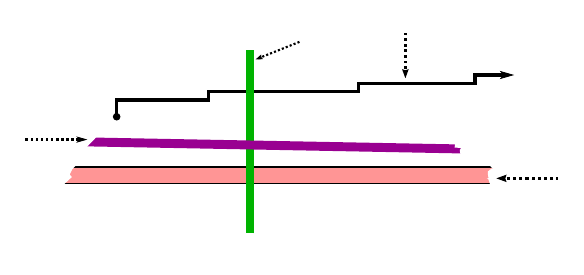}
\put(97,15){$\mathbf H$}
\put(68,43){$\gamma$}
\put(0,22){$A$}
\put(52,40){$B$}
\end{overpic}
\caption{A heuristic picture of the carrier $\mathbf H$ and the ray $\gamma$. Each hyperplane crossing $\gamma$ crosses $\mathbf H$ (for example, the hyperplane labeled $B$), and finitely many hyperplanes, like $A$, separate $\mathbf H$ from $\gamma$.}\label{fig:fellowtravel}
\vspace{-10pt}
\end{figure}

Conversely, if $\gamma:[0,\infty)\rightarrow\mathbf H$ is a combinatorial geodesic ray, then $\gamma\subset\mathbf C\cap\mathbf C^*$, so that $v\subset E\cap E^*$ by definition.

\textbf{The additional simplices:}  To complete the proof of the lemma, it therefore suffices to show that there exists a 0-simplex $v\in\simp\mathbf X$ such that $E=A\star v$, and a 0-simplex $v^*\neq v$ such that $E^*=A\star v^*$.  From this, the above discussion shows that $\simp\mathbf X=A\star(v\sqcup v')\cong\simp\mathbf H\star\mathbf Q_1$.

\textbf{The $T_{n-1}$-invariant hyperplane $X\subset\Euclidean^n$:}
Let $t_1,\ldots,t_{n-1}$ be a set of linearly independent translations that generate $T_{n-1}$, and let $h_0\in\mathbf H$.  Let $X\cong\reals^{n-1}$ be the $T_{n-1}$-invariant affine subspace of $\Euclidean^n$ containing the points $\eta(h_0),t_1(\eta(h_0)),\ldots,t_{n-1}(\eta(h_0))$.

There exists $S<\infty$ such that $X\subseteq N_S(\eta(\mathbf H))$ and $\eta(\mathbf H)\subseteq N_S(X)$.  Indeed, there exists $r_1\geq 0$ such that $d_{\mathbf X}(h,T_{n-1}(h_0))\leq r_1$ for all $h\in\mathbf H$, since $T_{n-1}$ acts cocompactly on $\mathbf H$.  Let $h\in\mathbf H$ and choose $t\in T_{n-1}$ such that $d_{\mathbf X}(h,t(h_0))\leq r_1$.  Since $t(h_0)\in X$, the distance in $\Euclidean^n$ from $\eta(h)$ to $X$ is at most
\[\|\eta(h)-\eta(th_0)\|\leq\lambda r_1+\mu.\]
Similarly, if $x\in X$, then the distance in $\Euclidean^n$ from $x$ to $\eta(\mathbf H)$ is at most
\[\|t\eta(h_0)-x\|\leq r_2,\]
where $r_2^2\geq\sum_{i=1}^{n-1}\|t_i\|^2$ and $t\in T_{n-1}$ is chosen, using the cocompactness of the $T_{n-1}$-action on $X$, to satisfy the preceding inequality.  Hence $S=\max\{\lambda r_1+\mu,r_2\}$ suffices.

Now let $t\in T_{_G}$ be a translation.  Then there exists $R<\infty$ such that each of $\mathbf H$ and $t(\mathbf H)$ lies in the uniform $R$-neighborhood in $\mathbf X$ of the other.  Indeed, since $\eta(\mathbf H)$ and $X$ lie in uniform $S$-neighborhoods of one another in $\Euclidean^n$, and $\eta$ is a quasi-isometry, it suffices to exhibit $R'<\infty$ such that $X$ and $t(X)$ lie in uniform $R'$-neighborhoods of one another.  Having shown that such an $R'$ exists, it is evident that $R=\lambda(R'+2S+\mu)$ suffices.  But since $t\in T_{_G}$, it is obvious that $X$ and $t(X)$ are parallel codimension-1 hyperplanes in $\Euclidean^n$, and so $R'\leq\|t(0)\|$.

\textbf{The translation $b$:}  Let $b\in T_{_G}-T_{n-1}$ be a translation chosen so that for all $x,x'\in X$, we have $\|b(x)-x'\|>2S.$  For example, let $b'$ be a translation along the unit normal vector to $X$ and let $b$ be some high power of $b'$.  If $\mathbf H\cap b(\mathbf H)\neq\emptyset$, then there exist $h,h'\in\mathbf H$ such that $h=t(h')$ and thus $\eta(h)=t\eta(h')$.  Now $b\eta(h')$ lies at distance at most $S$ from $b(X)$, and $\eta(h)$ lies at distance at most $S$ from $X$, so that $\|b(x)-x'\|\leq 2S$, a contradiction.  Thus $b(\mathbf H)$, the carrier of $b(H)$, is disjoint from $\mathbf H$ and lies in $\mathfrak h(H)$ or $\mathfrak h^*(H)$.  We can assume the former, by replacing $b$ with $b^{-1}$ if necessary.  Moreover, there exists $Q<\infty$ such that for all $h\in\mathbf H$, there exists $h',h''\in\mathbf H$ with $d_{\mathbf X}(h,t(h'))\leq Q$ and $d_{\mathbf X}(t(h),h'')\leq Q$.

Hence $\{b^p(H)\}_{p\geq 0}$ is an infinite collection of hyperplanes in $\mathbf C$ whose carriers are pairwise disjoint, such that $b^p(\mathbf H)$ lies in the $Q$-neighborhood of $b^{p+1}(\mathbf H)$ for all $p\geq 0$, and vice versa.  Moreover, by passing to a high power if necessary, we can assume that $b$ is a hyperbolic isometry of $\mathbf X$, and thus that $b(\mathfrak h(H))\subset\mathfrak h(H)$.  Therefore, $\{b^p(\mathfrak h(H))\}_{p\geq 0}$ is totally ordered by inclusion: $b^{p+1}(\mathfrak h(H))\subset b^{p}(\mathfrak h(H))$ for all $p\geq 0$, so that, if $p<q<r$, then $b^q(H)$ separates $b^p(H)$ and $b^r(H)$.  See Figure~\ref{fig:nested}.

\begin{figure}
  \begin{overpic}[width=0.2\textwidth]{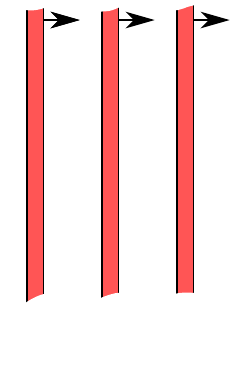}
   \put(5,10){$\mathbf H$}
   \put(25,10){$b(\mathbf H)$}
   \put(45,10){$b^2(\mathbf H)$}
  \end{overpic}

  \caption{The carrier $\mathbf H$ and some of its $\langle b\rangle$-translates.  The arrow on each translate of $\mathbf H$ points into the corresponding translate of $\mathfrak h(H)$.}\label{fig:nested}
  \vspace{-10pt}
\end{figure}

\textbf{Intersections of halfspaces:}  Let $p<q$ be integers, and consider the smallest subcomplex $\mathbf Y$ of $\mathbf X$ containing the region $b^p(\mathfrak h(H))\cap b^q(\mathfrak h^*(H))$ ``between'' $b^p(H)$ and $b^q(H)$.  Since $\eta$ is a quasi-isometry sending $b^p(H)$ and $b^q(H)$ to parallel hyperplanes in $\Euclidean^n$, we see that $\mathbf Y$ is contained in the regular $Q(q-p)$-neighborhood of $\mathbf H$.  This fact plays a role below.

\textbf{The 0-simplex $v$:}  For each $p\geq 0$, let $\beta_p$ be a combinatorial geodesic segment joining some $a_p\in\mathbf H^0\cap\mathfrak h^*(H)$ to some $c_p\in b^p(\mathbf H^0)\cap b^p\mathfrak h^*(H)$.  Let $a_p,c_p$ and $\beta_p$ be chosen so that $\beta_p$ is a short as possible.  Then the set of hyperplanes dual to 1-cubes of $\beta_p$ consists of $H$, together with those hyperplanes that separate $H$ from $b^p(H)$, because $H$ is convex.  In particular, $\beta_p$ contains a 1-cube dual to $b^q(H)$ if and only if $0\leq q<p$.

Now, since $T_{n-1}$ acts cocompactly on $\mathbf H$, there exists a finite set $\mathcal F$ of $0$-cubes in $\mathbf H$ such that, for all $p\geq 0$, we can choose $\beta_p$ in such a way that $\beta_p(0)\in\mathcal F$. Hence, by K\"{o}nig's lemma, there exists a combinatorial geodesic ray $\beta:[0,\infty)\rightarrow\mathbf C$ such that the initial 1-cube of $\beta$ is dual to $H$, and $\beta$ contains a 1-cube dual to $b^p(H)$ for all $p\geq 0$, and the hyperplane $U$ crosses $\beta$ if and only if $U=H$ or $U$ separates two elements of $\{b^p(H)\}_{p\geq 0}$.  The latter property implies that $[\beta]$ is minimal, and hence the simplex $v$ of $E$ represented by $\beta$ is a 0-simplex.

Since $\beta$ contains 1-cubes dual to infinitely many hyperplanes that do not cross $H$, $v\in E-A$.  See Figure~\ref{fig:betabuild}.  Note also that $bv=v$, since $\mathcal W(\beta)$ consists of the $b$-almost-invariant set $\{b^p(H)\}_{p\geq 0}$ together with any hyperplane separating two elements of that set.  In particular, no hyperplane crossing $H$ can cross $\beta$.  However, $\beta$ itself need not lie on a combinatorial geodesic axis for $b$.

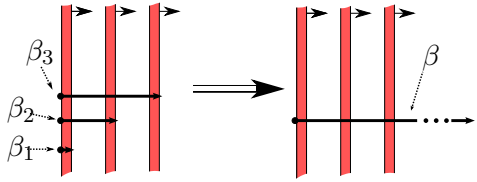
\begin{figure}
  \begin{overpic}[width=0.4\textwidth]{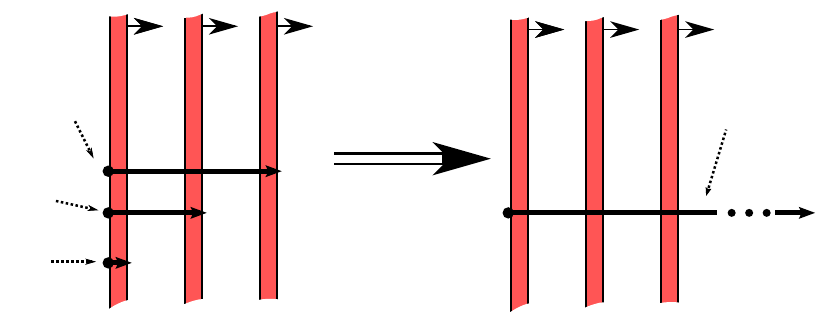}
    \put(2,7){$\beta_1$}
    \put(2,15){$\beta_2$}
    \put(6,27){$\beta_3$}
    \put(89,26){$\beta$}
  \end{overpic}
  \caption{At left are some of the segments $\beta_p$, shown in bold.  As at right, cocompactness of the $T_{n-1}$-action on $\mathbf H$ guarantees that there is a geodesic ray $\beta$ containing some $\beta_p$ for all $p\geq 0$.}\label{fig:betabuild}
  \vspace{-10pt}
\end{figure}

\textbf{Proof that $E\subseteq A\star v$:}  Let $\gamma:[0,\infty)\rightarrow\mathbf C$ be a combinatorial geodesic ray with $\gamma(0)=\beta(0)\in\mathbf H$.  Denote by $u$ the simplex of $E$ represented by $\gamma$.  We must verify that $u\subseteq A\star v$.  Now either $u\subset A$, or $b^p(H)$ crosses $\gamma$ for all sufficiently large $p\geq 0$.  Indeed, either $b^p(H)$ crosses $\gamma$ for all $p\geq 0$, or there exists $\pi\geq 0$ such that $b^{\pi+1}(H)$ does not cross $\gamma$.  In the latter case, $\gamma\subset\mathbf C\cap b^{\pi+1}(\mathfrak h^*(H))$, i.e. $\gamma$ lies ``between'' $\mathbf H$ and $b^{\pi+1}(\mathbf H)$.  It was shown above that the smallest subcomplex containing $\mathfrak h(H)\cap b^{\pi+1}(\mathfrak h^*(H))$ lies in $N_{Q(\pi+1)}(\mathbf H)$.  Hence $\gamma\subset N_{Q(\pi+1)}(\mathbf H)$ and it was shown above that this implies that $u$ is a simplex of $A$ and thus lies in $A\star v$.  See Figure~\ref{fig:liesbetween}.
\begin{figure}[h]
  \includegraphics[width=0.12\textwidth]{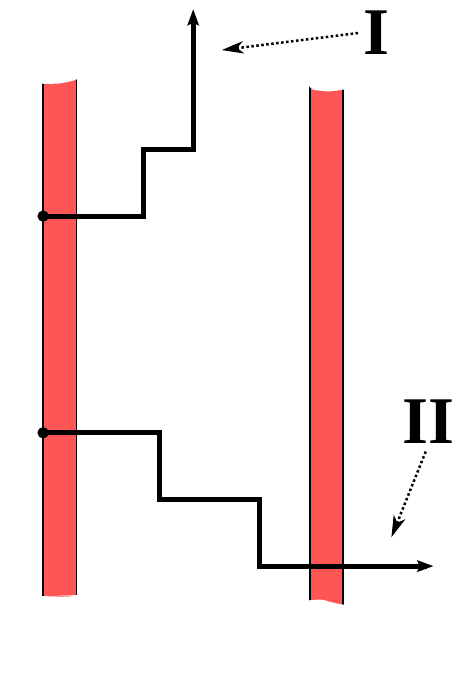}\\
  \caption{The ray $\mathbf I$ crosses $\mathbf H$ but fails to cross all but finitely many $\langle b\rangle$-translates of $H$, and hence lies in a uniform neighborhood of $\mathbf H$.  The ray labeled $\mathbf{II}$ could cross all $b^p(H)$.}\label{fig:liesbetween}
  \vspace{-5pt}
\end{figure}

Hence suppose that $b^p(H)$ crosses $\gamma$ for all $p\geq 0$.  Then the set $\mathcal W(\gamma)\cap\mathcal W(\beta)$ is infinite.  By Lemma~\ref{lem:raybuild}, there exists a geodesic ray $\sigma$ such that $[\sigma]\leq[\beta]$ and $[\sigma]\leq[\gamma]$.  Since $[\beta]$ is minimal, $[\sigma]=[\beta]$ and thus $[\beta]\leq[\gamma]$, i.e. $v\subseteq u$.  In particular, if $u$ is a 0-simplex, then $u=v$.  This shows that $v$ is the only 0-simplex of $E-A$, whence $E$ is isomorphic to a subcomplex of $A\star v$.

\textbf{Proof that $A\star v\subseteq E$:}  Let $a$ be a (visible) simplex of $A$.  First, using Lemma~\ref{lem:folding}, choose a combinatorial geodesic ray $\gamma:[0,\infty)\rightarrow\mathbf H$, representing $a$, so that $\gamma(0)=\beta(0)$.  Next, since $\mathbf H\cong[-\frac{1}{2},\frac{1}{2}]\times H$, we can modify $\gamma$ within its almost-equivalence class, without changing $\gamma(0)$, so that $\gamma\subset\mathfrak h^*(H)$.  Indeed, if $H_-\cong\{-\frac{1}{2}\}\times H$ is the copy of $H$ on the $\mathfrak h^*(H)$ side of $H$, we can if necessary project $\gamma$ to $H_-$.

Consider the combinatorial geodesic ray $\gamma'_p=b^p(\gamma)$.  Let $x_p$ be the initial 0-cube of the 1-cube of $\beta$ dual to $b^p(H)$, and let $\gamma_p:[0,\infty)\rightarrow b^p(\mathbf H)$ be a combinatorial geodesic ray such that $[\gamma_p]=[\gamma]$ and $\gamma_p(0)=x_p$.  For example, $\gamma_p$ can be produced by applying Lemma~\ref{lem:folding} to $b^p(\gamma)$ and the sub-ray of $\beta$ beginning at $x_p$.  It is easily seen that $\gamma$ and $b^p(\gamma)$ fellow-travel, since $b$ acts as a translation on $\Euclidean^n$.

Let $\mathcal U_p\subset\mathcal W(\gamma)$ be the set of hyperplanes $U$ that cross $\gamma$ and do not cross $b^{p+1}(H)$, and let $\mathcal U=\bigcup_{p\geq 0}\mathcal U_p$.  By Lemma~\ref{lem:Ufinite}, $|\mathcal U|<\infty$ and thus, by Lemma~\ref{lem:newchoice}, we have rays $\beta',\gamma'$ such that $[\beta']=[\beta],[\gamma']=[\gamma]$, and $\gamma'(0)=\beta'(0)$, and $\mathcal W(\gamma')\subseteq\mathcal W(\gamma)-\mathcal U$.  Now, if $U\in\mathcal W(\gamma')$, then $U$ crosses each $b^p(H)$, and hence $U$ crosses each $V\in\mathcal W(\beta')$.  By Lemma~\ref{lem:raybuild2}, there exists a geodesic ray $\sigma$ such that $[\beta]\leq[\sigma]$ and $[\gamma]\leq[\sigma]$.  Hence $a$ and $v$ are both contained in the simplex $u$ of $E$ represented by $\sigma$, whence there is a simplex $a\star v\subseteq E$.  Hence $A\star v\subseteq E$.

\textbf{Conclusion:}  We have shown that $E=A\star v$.  Arguing in the same way in $\mathbf C^*$ shows that there is a 0-simplex $v^*$ of $E^*-A$ such that $E^*=A\star v^*$.  Hence $\simp\mathbf X=E\cup E^*\cong A\star(v\sqcup v^*)$ and the proof is complete.
\end{proof}

\begin{lem}\label{lem:Ufinite}
$\mathcal U$ is finite.
\end{lem}

\begin{proof}
For all $U_i\in\mathcal U_p$ and all $t\in T_{n-1}$, the hyperplane $tU_i$ crosses $tH=H$, since $U_i$ crosses $H$ by virtue of being dual to a 1-cube of $\gamma$.  On the other hand, $tU_i$ does not cross $tb^p(H)=b^p(H)$, since $U_i$ does not cross $b^p(H)$.  Thus $tU_i\in\mathcal U_p$ for all $t\in T_{n-1}$.

Since $\gamma\subset N_{pQ}(b^p(\mathbf H))$ for each $p\geq 0$, we have that $|\mathcal U_p|<\infty$ for all $p\geq 0$.  Therefore, if $\mathcal U$ is infinite, then for all $p\geq 0$, there is a hyperplane $U^p$ dual to a 1-cube of $\gamma$ that crosses $b^p(H)$ and does not cross $b^{p+1}(H)$.  By cocompactness of the $T_{n-1}$-action on $\mathbf H$, there is a translation $t_p\in T_{n-1}$ such that $t_pU^p$ is dual to a 1-cube $c_p$ that lies within some fixed distance $f$ of $\gamma(0)$.  But then the $f$-neighborhood of $\gamma(0)$ in $\mathbf H$ contains 1-cubes dual to infinitely many $T_{n-1}$-distinct hyperplanes, contradicting local finiteness of $\mathbf X$.  Hence $\mathcal U$ is finite, and there exists $\pi\geq 0$ such that $\mathcal U_p=\mathcal U_{\pi}$ for all $p\geq\pi$.
\end{proof}

In the next lemma, we use disc diagrams in $\mathbf X$.  Discussions of minimal-area diagrams in the same language as is used here can be found in~\cite{HagenQuasiArb} and~\cite{WiseIsraelHierarchy}.  The lemma is stated in more generality than required in the present context, but it is easy to verify that the hypotheses are satisfied by $\mathbf X,\mathbf H,\beta,\gamma$ from the proof of Lemma~\ref{lem:induction1}.

\begin{lem}\label{lem:newchoice}
Let $\mathbf X$ be a CAT(0) cube complex containing a hyperplane $H$, with carrier $\mathbf H$.  Suppose that there is a hyperbolic element $b\in\Aut(\mathbf X)$ such that $\langle b\rangle H$ is a pairwise-disjoint collection of hyperplanes.  For each $p\in\integers$, let $\mathfrak h(b^p(H))$ be the halfspace of $\mathbf X$ associated to $b^p(H)$ that contains $b^{p+1}(H)$ and let $\mathfrak h^*(b^p(H))$ be the complementary halfspace.

Let $\gamma\rightarrow\mathbf H$ be a combinatorial geodesic ray in $\mathfrak h^*(H)$, and let $\beta\rightarrow\mathbf X$ be a combinatorial geodesic ray such that $\beta(0)=\gamma(0)$ and such that the hyperplane $b^p(H)$ crosses $\beta$ for $p\geq 0$, and the hyperplane $U$ crosses $\beta$ only if $U=H$ or $U$ separates two elements of $\langle b\rangle H$.

For each $p\geq 0$, let $\mathcal U_p$ be the set of hyperplanes that cross $\gamma$ but not $b^{p+1}(H)$, and let $\mathcal U=\cup_{p\geq 0}\mathcal U_p$.  Finally, suppose that there exists $\pi\geq 0$ such that $\mathcal U_p=\mathcal U_\pi$ for $p\geq \pi$.

There exist rays $\beta',\gamma'$ such that $[\beta']=[\beta],[\gamma']=[\gamma]$, and $\gamma'(0)=\beta'(0)$, and $\mathcal W(\gamma')\subseteq\mathcal W(\gamma)-\mathcal U$, and $\mathcal W(\beta')\subseteq\mathcal W(\beta)$.
\end{lem}
\begin{proof}
We shall modify $\gamma$ within its almost-equivalence class, without changing its initial 0-cube, to produce a ray $\gamma'$ with $\mathcal W(\gamma')=\mathcal W(\gamma)-\mathcal U$, and take $\beta'=\beta$, reaching a conclusion that is stronger than the statement of the lemma (which is all we need).  If $\mathcal U=\emptyset$, we are done.

First, choose any $p\geq\pi$ and choose $\delta$ sufficiently large that for all $U\in\mathcal U$, the 1-cube of $\gamma$ dual to $U$ is $\gamma([k,k+1])$, where $k\leq\delta-1$.  For any $q>\delta$, let $\widehat{\gamma}=\gamma([0,q])$ and let $\widehat{\beta}=\beta([0,p])$.  The hyperplane $V$ dual to the terminal 1-cube of $\widehat{\gamma}$ necessarily belongs to $\mathcal W(\gamma)-\mathcal U$, and therefore crosses $b^p(H)$.  Let $\widehat{\gamma_p}$ be a shortest geodesic segment in $b^p(\mathbf H)$ joining $x_p$ to a 0-cube of $b^p(\mathbf H)\cap N(V)$, and let $\omega$ be a shortest geodesic segment in $N(V)$ joining $\gamma(q)$ to the terminal 0-cube of $\widehat{\gamma_p}$.

Let $D\rightarrow\mathbf X$ be a minimal-area disc diagram bounded by $\widehat\gamma\omega(\widehat{\gamma_p})^{-1}(\widehat{\beta})^{-1}$, with $\omega$ chosen so as to minimize the area of $D$ among all such diagrams with $\widehat{\gamma},\widehat{\gamma_p},\widehat{\beta}$ fixed and $\omega$ allowed to vary.  Let $K$ be a dual curve in $D$ emanating from $\widehat{\gamma}$ and mapping to a hyperplane $U\in\mathcal U$.  Since $\gamma$ is a geodesic, $K$ cannot end on $\widehat{\gamma}$.  Since $U$ does not cross $b^p(H)$, $K$ cannot end on $\widehat{\gamma_p}$.  Since $U$ is not dual to a 1-cube of $\beta$ -- observe that $\mathcal W(\beta)\cap\mathcal W(\gamma)=\emptyset$ since no hyperplane crossing $H$ crosses $\beta$ -- the unique possibility is that $K$ ends on $\omega$, and hence $U$ crosses $V$.  Such a dual curve $K$ is shown in Figure~\ref{fig:discdiagram}.  Let $\zeta_q$ be the path in $D$ which travels along $\gamma$, starting at $\gamma(0)$, until reaching the initial 0-cube of the 1-cube dual to $K$, and then travels geodesically along a
path in the carrier of $K$, until reaching $\omega$, in such a way as never to cross $K$.  See Figure~\ref{fig:discdiagram}.
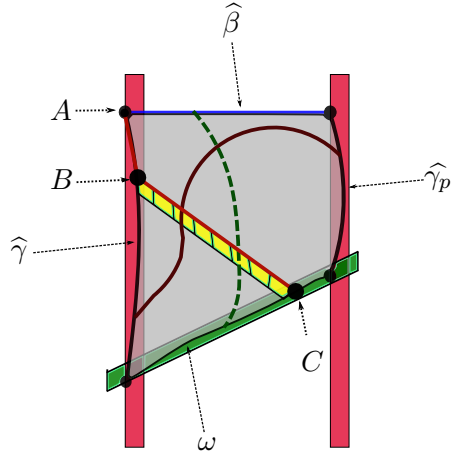
\begin{figure}[h]
\begin{overpic}[width=0.4\textwidth]{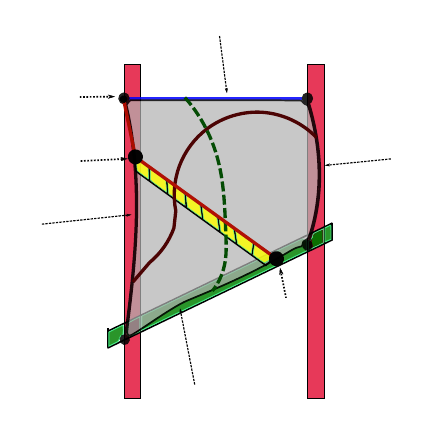}
\put(5,48){$\widehat\gamma$}
\put(89,63){$\widehat{\gamma_p}$}
\put(13,76){$A$}
\put(13,62){$B$}
\put(64,25){$C$}
\put(48,94){$\widehat{\beta}$}
\put(43,9){$\omega$}
\end{overpic}
\caption{The diagram $D$.  The rectangular ladder is the carrier of the dual curve $K$; the path $ABC$ is the path $\zeta_q$.  No dual curve emanating from $\omega$ can cross $K$, as the dashed one does; hence every dual curve crossing $\zeta_q$ ends, like the solid one, on $\widehat{\gamma}$.}\label{fig:discdiagram}
\end{figure}
Let $\gamma'_q$ be the image in $\mathbf X$ of $\zeta_q$ under the map $D\rightarrow\mathbf X$.  Any hyperplane crossing $\gamma'_q$ corresponds to a dual curve in $D$ that crosses $\zeta_q$.  Now, by minimality of the area of $D$, every dual curve crossing $\zeta_q$ crosses $\gamma$, and hence any two such dual curves map to distinct hyperplanes.  Thus $\gamma'_q$ is a geodesic segment.  Moreover, any dual curve in $D$ that travels from $\widehat{\gamma}$ to $\widehat{\gamma_p}$ is necessarily dual to a 1-cube of $\zeta_q$.  Thus every $W\in\mathcal W-\mathcal U$ crosses $\gamma'_q$ for all sufficiently large $q$.  Applying K\"{o}nig's lemma to the set of all $\gamma'_q$, as $q$ grows arbitrarily large, yields a geodesic ray $\gamma'$ such that $\gamma'(0)=\beta(0)$ and $\mathcal W(\gamma)-\mathcal U\subseteq\mathcal W(\gamma')\subseteq\mathcal W(\gamma)-\{U\}$.  Since $\mathcal U$ is finite, we reach the desired $\gamma'$ after finitely many repetitions of this argument.
\end{proof}

\subsection{Application to actions on $\mathbf R_n$}\label{sec:actionsonrn}
The next theorem is an application of Theorem~\ref{thm:product_action}, but is independent of the results in the next section.  It is proved by different means in~\cite{WiseIsraelHierarchy}.

\begin{thm}\label{thm:product_action}
For $n\geq 1$, let $G$ be a virtually $\integers^n$ group acting properly and cocompactly on a CAT(0) cube complex.  Then $G$ acts properly and cocompactly on $\mathbf R_n$.
\end{thm}

\begin{proof}
The proof has several steps; the first step could also be accomplished using the rank-rigidity theorem of~\cite{CapraceSageev}; here we describe a proof using a fact about the simplicial boundary.

\textbf{Decomposing $\mathbf X$ as a cubical product:} Without loss of generality, $G$ acts essentially on $\mathbf X$.  Hence Theorem~\ref{thm:boundaryoctahedron}, together with Theorem~3.30 of~\cite{HagenBoundary}, implies that $\mathbf X\cong \mathbf X_1\times\ldots\times\mathbf X_n$, where each $\mathbf X_i$ is a CAT(0) cube complex quasi-isometric to $\mathbf R_1$, and the inclusion $\mathbf X_i\hookrightarrow\mathbf X$ induces the inclusion of a copy of a factor $\mathbf Q_1\subset\simp\mathbf X\cong\mathbf Q_1\star\ldots\star\mathbf Q_1$.  Indeed, $\simp\mathbf X\cong\mathbf Q_n$ by Theorem~\ref{thm:boundaryoctahedron}; to apply~\cite[Theorem~3.30]{HagenBoundary} then requires only that each simplex of $\simp\mathbf X$ be visible.  However, the proof of Theorem~\ref{thm:boundaryoctahedron} shows that each 0-simplex of $\simp\mathbf X$ arises from a combinatorial geodesic ray, as required.

Regarding $\mathbf X$ as the above product, we choose a basepoint $x_i\in\mathbf X_i$ for each $i$.  Abusing notation, we shall refer to the subcomplex $\mathbf X_i\times\prod_{i\neq j}\{x_j\}\subset\mathbf X$ as $\mathbf X_i$.

\textbf{$\Aut(\mathbf X_i)$-invariant lines:}  For each $i\leq n$, let $G_i=\stabilizer_{\mathbf X}(\mathbf X_i)$.  Since $G_i$ acts essentially on $\mathbf X_i$ and $\mathbf X_i$ is quasi-isometric to $\mathbf R_1$, any three pairwise-disjoint hyperplanes of $\mathbf X_i$ have the property that one separates the other two.  This, together with the fact that $G_i$ acts properly and cocompactly on $\mathbf X_i$, allows us to invoke~\cite[Theorem~7.2]{CapraceSageev} and conclude that there is a (not necessarily combinatorial) CAT(0) geodesic line $\alpha_i\subseteq\mathbf X_i$ that is $G_i$-invariant (and in fact $\Aut(\mathbf X_i)$-invariant).

\textbf{$\Aut(\mathbf X)$-invariant $\mathbf R_n$:}  Without loss of generality, the basepoint was chosen in each factor so that for each $i$, we have $x_i\in\alpha_i$.  Thus $\mathbf X$ contains a flat $\prod_i\alpha_i$ that contains the point $x=(x_1,\ldots,x_n)$.  By~\cite[Proposition~2.6]{CapraceSageev}, $\Aut(\mathbf X)$ preserves the product decomposition of $\mathbf X$, possibly permuting isomorphic factors.  Since each $\alpha_i$ is $\Aut(\mathbf X_i)$-invariant, $\prod_i\alpha_i$ is $\Aut(\mathbf X)$-invariant.  Declare each point in the $\Aut(\mathbf X)$-invariant set $\prod_{i=1}^n\left(\Aut(\mathbf X_i)x\right)$ to be a 0-cell, with two 0-cells adjacent if and only if they belong to a common $\Aut(\mathbf X)$-translate of some $\alpha_i$.  The resulting $\Aut(\mathbf X)$-invariant graph is easily seen to be the 1-skeleton of $\mathbf R_n$.  We thus have a $G$-invariant (not necessarily combinatorial) embedded copy of $\mathbf R_n$ in $\mathbf X$.  Since $G$ acts properly and cocompactly on $\mathbf X$, the action on $\mathbf R_n$ is proper and cocompact.
\end{proof}


\section{The action of $G$ on $\simp\mathbf X$}\label{sec:actiononboundary}
Let $G$ be an $n$-dimensional crystallographic group acting properly and cocompactly on the CAT(0) cube complex $\mathbf X$ and recall that $\eta:\mathbf X\rightarrow\Euclidean^n$ denotes a $G$-equivariant, $(\lambda,\mu)$-quasi-isometry.  By Theorem~\ref{thm:boundaryoctahedron} and Proposition~\ref{prop:inducedaction}, there is an exact sequence
\[1\rightarrow K\rightarrow G\rightarrow\Aut(\mathbf Q_n),\]
where $K$ is the normal subgroup of $G$ consisting of those elements that act as the identity on $\simp\mathbf X\cong\mathbf Q_n$.    The main theorem of this section is:

\begin{thm}\label{thm:main}
Let $G$ be an $n$-dimensional crystallographic group, with $n\geq 1$.  If $G$ acts properly and cocompactly on a CAT(0) cube complex $\mathbf X$, then there is an exact sequence
\[1\rightarrow\integers^n\rightarrow G\rightarrow\Aut(\mathbf Q_n),\]
and, moreover, $G$ is hyperoctahedral.
\end{thm}

Before proving Theorem~\ref{thm:main}, we require two lemmas:

\begin{lem}\label{lem:translationsfixboundary}
$T_{_G}\leq K$.
\end{lem}

\begin{proof}
Let $t\in T_{_G}$ and let $\gamma:[0,\infty)\rightarrow\mathbf X$ be a combinatorial geodesic ray.  Now, for all $s\geq 0$, we have $d_{\mathbf X}(\gamma(s),t(\gamma(s))\leq\lambda\left(\|\eta(\gamma(s))-t(\eta(\gamma(s)))\|+\mu\right)\leq\lambda\left(\|t(0)\|+\mu\right).$ Thus $\gamma$ and $t(\gamma)$ fellow-travel in $\mathbf X$, so that $\mathcal W(\gamma)\triangle\mathcal W(t(\gamma))$ is finite, i.e. $[\gamma]=t[\gamma]$.
\end{proof}

\begin{lem}\label{lem:torsionfreecase}
Let $g\in G$ fix $\simp\mathbf X$.  Then $g\in T_{_G}$.
\end{lem}

\begin{proof}
Let $H,\mathbf H,\mathbf C,\mathbf C^*\subset\mathbf X$ and $A,E,E^*,v,v^*\subset\simp\mathbf X$ be as in the proof of Lemma~\ref{lem:induction1}, so that $\simp\mathbf X\cong A\star(v\sqcup v^*)$.  If $n=1$, then $A=\emptyset$ and either $g\in G$ acts as a translation in $\reals$, or $g$ exchanges $v$ and $v^*$.  Hence $K\leq T_{_G}$.

Suppose that $n\geq 2$, so that $A\neq\emptyset$, and let $g\in K$.  By induction, if $g\in G_{n-1}\cap K$, then $g\in T_{n-1}$.  Hence suppose that $gH\neq H$.  Now, $g$ fixes $A$ and fixes the two 0-simplices $v,v^*$.  Recall that $X\subset\reals$ is a copy of $\reals^{n-1}$ whose stabilizer is $G_{n-1}$ and which lies at finite Hausdorff distance from $\eta(\mathbf H)$.  Assume that $g(X)\neq X$ and that $\psi(g)$ is a non-identity orthogonal transformation.  Now, if $g(X)$ is not parallel to $X$, then applying the quasi-inverse of $\eta$ shows that for each $N$, there exists $h_N\in\mathbf H$ such that $d_{\mathbf X}(\mathbf H,g(h_N))\geq N$.  Therefore, by cocompactness, there exists a combinatorial geodesic ray $\gamma$ in $\mathbf H$ such that for all $N\geq 0$, there exists $s_N\geq 0$ for which $d_{\mathbf X}(g\gamma(s_N),\mathbf H)\geq N$.  Hence $\gamma$ is crossed by infinitely many hyperplanes that do not cross $\mathbf H$, and thus $\gamma$ does not represent a simplex in $A$, a contradiction.
Hence $g\in T_{_G}$.

If $X$ and $g(X)$ are parallel, then $\mathbf H$ and $g(\mathbf H)$ lie at finite Hausdorff distance in $\mathbf X$, and, since $g$ fixes $A$, every minimal geodesic ray in $\mathbf H$ fellow-travels with its $g$-translate.  Applying $\eta$ shows that $g$ moves every point in $\Euclidean^n$ a uniformly bounded distance, whence $g\in T_{_G}$.
\end{proof}

We are now ready to prove Theorem~\ref{thm:main}.  The second assertion could also be deduced from Theorem~\ref{thm:product_action}, but here we deduce it directly from Theorem~\ref{thm:boundaryoctahedron}, avoiding the action on $\mathbf R_n$.

\begin{proof}[Proof of Theorem~\ref{thm:main}]
\textbf{The exact sequence:}  By Lemma~\ref{lem:translationsfixboundary}, $T_{_G}\leq K$.  By Lemma~\ref{lem:torsionfreecase}, $K\leq T_{_G}$, whence $K\cong T_{_G}\cong\integers^n$ and $P_{_G}\cong G/K$ is isomorphic to a subgroup of $\Aut(\mathbf Q_n)$.  It remains to verify that $G$ is hyperoctahedral.

\textbf{A monomorphism $P_{_G}\rightarrow O(n,\integers)$:}  Let $I:G\rightarrow\Aut(\simp\mathbf X)\cong\Aut(\mathbf Q_n)$ be the induced action on the simplicial boundary, whose kernel is $T_{_G}$.  Note that for a simplex $v$ of $\simp\mathbf X$ corresponding to an almost-equivalence class $[\gamma]$, the simplex $I(g)(v)$ corresponds to $[g(\gamma)]$.

Let $\{\pm v_i\}_{i=1}^n$ be the set of 0-simplices of $\simp\mathbf X$, labeled so that for each $i$, the simplex $-v_i$ is the unique 0-simplex that is not adjacent to $+v_i$.  A simplex $u$ of $\simp\mathbf X$ is uniquely expressible as a vector $\vec u=(z_{u,i})_{i=1}^n$ with $z_{u,i}\in\{-1,0,1\}$, where $z_{u,i}=\pm1$ exactly when $\pm v_i\in u$, and 0 otherwise.  For each $g\in G$, we represent $I(g)$ as an $n\times n$ signed permutation matrix $M_g$, so that ${I(g)(\vec u)}=M_g\vec u$.  The map $\psi(g)\mapsto M_g$ defines a monomorphism $\iota:P_{_G}\rightarrow O(n,\integers)$.

\textbf{A monomorphism $O(n,\integers)\rightarrow O(n,\reals)$:}  Let $\eta':\Euclidean^n\rightarrow\mathbf X$ be a $G$-equivariant quasi-inverse for $\eta$.  For any geodesic ray $L:[0,\infty)\rightarrow\Euclidean^n$ with $L(0)=0$, there exists a combinatorial geodesic ray $\gamma_L:[0,\infty)\rightarrow\mathbf X$ such that $\eta'(L)$ and $\gamma_L$ fellow-travel at distance $\kappa$ for some $\kappa$ independent of $L$.  Let $v_L$ be the simplex of $\simp\mathbf X$ represented by $\gamma_L$.  Note that, since combinatorial geodesics that fellow-travel are almost-equivalent, $v_L$ does not depend on the particular choice of $\gamma_L$.

Now, for any $L$ and any $g\in G$, let $L'=\bar{\theta}(\psi(g))(L)+\tau_g$, so that $g\eta'(L)=\eta'(gL)=\eta'(\bar{\theta}(\psi(g))(L)+\tau_g)=\eta'(L')$, which fellow-travels with $g\gamma_L$ and $\gamma_{L'}$, where $\gamma_{L'}$ is a combinatorial geodesic ray that fellow-travels with $\eta'(\bar{\theta}(\psi(g))(L))$.  Thus $[\gamma_{L'}]=[g(\gamma_L)]$, i.e. $I(g)(v_L)=[g(\gamma_{L'})]$

We now define a monomorphism $\rho:O(n,\integers)\rightarrow O(n,\reals)$ such that $\rho\circ\iota=\bar{\theta}$.  Let $M\in O(n,\integers)$ be a signed permutation matrix.  For each $i$, let $\gamma_i^+$ be a combinatorial geodesic ray representing $+v_i$ and define $\gamma_i^-$ likewise for $-v_i$.  Since $v_i^+$ and $v_i^-$ are non-adjacent 0-simplices, these rays can be chosen so that their union is equal to the image of a combinatorial geodesic $\sigma_i:\reals\rightarrow\mathbf X$ (see~\cite[Theorem~3.24]{HagenBoundary}).

For each $i$, there exists a unique (oriented) line $L_i$ in $\Euclidean^n$ such that $L_i(0)=0$ and $\eta(\sigma_i)$ fellow-travels with $L_i$.  Uniqueness is obvious, since any two such lines $L_i$ and $L^{\circ}_i$ must fellow-travel, and any two distinct lines through the origin in $\Euclidean^n$ either coincide or fail to fellow-travel.  We now construct $L_i$.  Let $x$ be the initial point of $\gamma_i$.  Without loss of generality, since $T_{_G}$ acts cocompactly, the hyperplane $H$ dual to the initial 1-cube of $\gamma_i$ is essential and has the property that for some $g\in T_{_G}$, the hyperplanes $H$ and $gH$ are disjoint and both cross $\gamma_i$.  Now, for $0\leq k<k'<k''$, the hyperplane $g^{k'}H$ separates $g^{k''}H$ from $g^kH$.  Otherwise, these three essential hyperplanes would form a facing triple leading to three pairwise-nonadjacent 0-simplices in $\simp\mathbf X$, which is easily seen to be impossible in a hyperoctahedron.  Hence $g^kH$ crosses $\gamma_i$ for all $k\geq 0$.  Applying this argument again on the other side of $H$ shows that $\sigma_i$ can be chosen to be an axis for $g$, and we take $L_i$ to be the line in $\Euclidean^n$ in the direction of $\tau_g$.

Let $\dot t_i$ be the unit vector in the positive $L_i$-direction.  Then the $\dot t_i$ form a basis for $\reals^n$.  Indeed, let $E\subseteq\Euclidean^n$ be the subspace spanned by $\{\dot t_i\}_{i=1}^n$.  If $E\subsetneq\Euclidean^n$, then there are 0-cubes in $\mathbf X$ arbitrarily far from $\eta'(E)$, and hence either there are either simplices of $\simp\mathbf X$ that are not spanned by $\{\pm v_i\}_{i=1}^n$ or some $+v_i$ is adjacent to $-v_i$.  Each of these situations is impossible, whence $\{\dot t_i\}_{i=1}^n$ is a basis.

Let $A\in GL(n,\reals)$ be the matrix so that $\dot t_i=A\dot e_i$ for each $i$, where $\{\dot e_i\}_{i=1}^n$ is the standard basis.  Then for each signed permutation matrix $M$, the matrix $AMA^{-1}$ is an isometry of $\Euclidean^n$ that acts as a permutation of $\{\pm \dot t_i\}_{i=1}^n$.  The matrix $A$ is uniquely determined by $\{\dot t_i\}$ and so $AMA^{-1}$ is uniquely determined by $M$ and the $\{\pm v_i\}_{i=1}^n$.  Let $\rho(M)=AMA^{-1}$; this defines a monomorphism $\rho: O(n,\integers)\rightarrow O(n,\reals)$.

\textbf{Conclusion:}  With $M$ as above, for each $i$, there is a combinatorial geodesic $\sigma_i'$ in $\mathbf X$ corresponding to the pair $(M(-v_i),M(+v_i))$, and $\eta(\sigma'_i)$ fellow-travels with a unique line $L'_i$ through the origin in $\Euclidean^n$.  Now, since $M$ is an automorphism of the boundary, and $L_i$ and $L'_i$ are uniquely determined by $+v_i$ and $M(\vec{+v_i})$, we see that $L'_i=L_j$ or $-L_j$ for some $j\leq n$.

Hence, for $g\in G$, we have $\rho(\iota(\psi(g)))=AM_gA^{-1}$.  On the other hand, for all $i$, we have that $AM_gA^{-1}(L_i)=\pm L_j$ for some $j$, and $\eta'(\pm L_j)$ fellow-travels with $\sigma_j$.  But $\eta'(\bar{\theta}(\psi(g))(L_i)$ fellow-travels with $\sigma_j$, which represents the simplices $\pm v_j=\iota(\psi(g))(\pm v_i)$.  Hence $\rho\circ\iota=\bar{\theta}$, and $\iota$ corresponds to conjugation by $A\in GL(n,\reals)$, i.e. $G$ is hyperoctahedral.\end{proof}
\section{Constructing actions on $\mathbf R_n$}\label{sec:cubulating}
\subsection{The standard cubulation of a crystallographic group}
The next lemma involves a well-known construction (see, e.g. Section~16 of~\cite{WiseIsraelHierarchy}), and we include a proof only for completeness. By a result of Zassenhaus (see~\cite{RatcliffeBook}), the conclusion of Lemma~\ref{lem:cubulating} also holds for torsion-free virtually free abelian groups, and this is the form in which it is given in~\cite{WiseIsraelHierarchy}.

\begin{lem}\label{lem:cubulating}
Let $G$ be an $n$-dimensional crystallographic group.  Then $G$ acts properly on the CAT(0) cube complex $\mathbf R_N$ for some $n\leq N\leq n|P_{_G}|$.
\end{lem}

\begin{proof}
Let $\{\vec t_i\}_{i=1}^n$ be a basis for $\reals^n$. For each $i$, let $X_i=\linspan\{\vec t_j\}_{j\neq i}$, which is invariant under the translation $t_j$ for all $j\neq i$.  Define a wall by declaring $\mathfrak h^*(X_i)$ to be a component of $\Euclidean^n-X_i$ that contains the origin, and $\mathfrak h(X_i)=X_i\cup(\Euclidean^n-\mathfrak h^*(X_i))$.  Let $\mathcal W$ be the set of walls consisting of all $G$-translates of these $n$ geometric walls.

$G$ acts on the cube complex $\mathbf X$ dual to the wallspace $\left(\Euclidean^n,\mathcal W\right)$, and it remains to verify that this action is proper.  For $r_1,r_2\in\Euclidean^n$, write $r_1-r_2=\sum_{i=1}^n\nu_i\vec t_i,$ where $\nu_1,\ldots,\nu_n\in\reals$.  For each $i$, at least $\left\lfloor|\nu_i|\right\rfloor\geq|\nu_i|-1$ distinct $T_{_G}$-translates of $X_i$ separate $r_1$ from $r_2$, so that $\#(r_1,r_2)\geq\sum_{i=1}^n|\nu_i|-n.$
Therefore,
\[\|r_1-r_2\|^2=\sum_{i=1}^n\nu_i^2\|\vec t_i\|^2\leq\max_i\|\vec t_i\|^2\left(\sum_{i=1}^n|\nu_i|\right)^2\leq\max_i\|\vec t_i\|^2\left(\#(r_1,r_2)+n\right)^2,\]
so that the wallspace $\left(\Euclidean^n-\bigcup_{i=1}^nG(X_i),\mathcal W\right)$ satisfies the linear separation property.  Since $G$ acts metrically properly on $\Euclidean^n$, it follows that $G$ acts properly on $\mathbf X$.

Note that, for each $i$ and each $t,t'\in T_{_G}$, the walls $t(X_i)$ and $t'(X_i)$ do not cross, and $T_{_G}(X_i)$ is a collection of parallel codimension-1 hyperplanes in $\Euclidean^n$ such that the cube complex dual to $\left(\Euclidean^n,T_{_G}(X_i)\right)$ is isomorphic to $\mathbf R_1$.  Moreover, for any $g,g'\in G$ and $i,j\leq n$, the walls $g(X_i)$ and $g'(X_j)$ are parallel if and only if $i=j$ and $g'g^{-1}\in T_{_G}$; otherwise, the corresponding hyperplanes of $\mathbf X$ cross.  Hence $\mathbf X$ is isomorphic to the product of $N$ copies of $\mathbf R_1$, i.e. $\mathbf X\cong\mathbf R_N$, where $N$ is equal to the cardinality of the image of $P_{_G}(\{\vec t_i/\|\vec t_i\|\}_{i=1}^n)$ under the map $S^{n-1}\rightarrow S^{n-1}/\integers_2$.
\end{proof}


\subsection{Cocompactness when the point group is hyperoctahedral}
The next theorem combines with Theorem~\ref{thm:main} to prove that cocompactly cubulated crystallographic groups act properly and cocompactly on $\mathbf R_n$, without using Theorem~\ref{thm:product_action}.

\begin{thm}\label{thm:cubulatingcrystallographic}
Let $G$ be an $n$-dimensional hyperoctahedral crystallographic group.  Then $G$ acts properly and with a single orbit of $n$-cubes on $\mathbf R_n$.
\end{thm}

\begin{proof}
Let $X_i$ be the codimension-1 subspaces given by applying the construction of walls in the proof of Lemma~\ref{lem:cubulating} to a basis $\{\vec t_i\}_{i=1}^n$ of $\Euclidean^n$ upon which $P_{_G}$ acts by signed permutations; such a basis exists by Lemma~\ref{lem:hyperoctahedralbasis}.  For any $g\in G$, we have $g(X_i)=\bar{\theta}(\psi(g))(X_i)+\tau_g$.  Now, since $\bar{\theta}(\psi(g))(\vec t_i)=\pm\vec t_j$ for some $j$, this implies that $g(X_i)\in T_{_G}(X_j)$.  Proceeding exactly as in the proof of Lemma~\ref{lem:cubulating} shows that $G$ acts properly on the CAT(0) cube complex $\mathbf R_n$ dual to the wallspace whose walls are $\bigcup_{i=1}^nT_{_G}(X_i)$.  It is easily verified that there is one $T_{_G}$-orbit of maximal families of pairwise-crossing walls, and hence $G$ acts on $\mathbf R_n$ with one orbit of $n$-cubes.
\end{proof}

\subsection{Stabilization}
From Lemma~\ref{lem:cubulating}, we get Corollary~C, by adding to the proof of Lemma~16.8 of~\cite{WiseIsraelHierarchy} the additional information that the groups involved are crystallographic.

\begin{cor}\label{cor:d1}
Let $G$ be an $n$-dimensional crystallographic group.  Then there exists $m\geq 0$ and $\phi:G\rightarrow GL(m,\integers)$ such that $\integers^m\rtimes_{\phi}G$ is a crystallographic group that acts properly and cocompactly on a CAT(0) cube complex.
\end{cor}

\begin{proof}
By Lemma~\ref{lem:cubulating}, there exists $N\geq n$ such that $G$ acts properly on $\mathbf R_N$.  Let $\ddot G=\Aut(\mathbf R_N)$, so that we have a homomorphism $G\rightarrow\ddot G$ with finite kernel.  Now, $\ddot G$ is the automorphism group of a cocompact lattice in $\Euclidean^N$, namely the 0-skeleton of $\mathbf R_N$, and thus $\ddot G$ is an $N$-dimensional hyperoctahedral crystallographic group.  Let $K=\ker(G\rightarrow\ddot G)$, and let $\mathbb F\cong\Euclidean^n$ be a $G$-invariant subspace of $\Euclidean^N$.  Since $K$ acts trivially on $\Euclidean^N$, the action of $K$ on $\mathbb F$ is trivial, and thus $K=\{1\}$ since $G$ acts faithfully on $\mathbb F$, being $n$-dimensional crystallographic.  Hence $G\leq\ddot G$.  Also, $\ddot G$ contains a maximal subgroup $S\cong\integers^m$ generated by $m=N-n$ linearly independent translations, each orthogonal to $\mathbb F$, such that $S\cap T_{_G}=\{1\}$.  Now $S\times T_{_G}\cong\integers^N$ is a subgroup of $\ddot G$ consisting of translations, and,
since $T_{_G}$ is normal in $G$, for any $s\in S$ and $g\in G$, we have that $gsg^{-1}\in S\cup T_{_G}$, since $g\in\ddot G$ and $\ddot G$ is crystallographic, and $gsg^{-1}\not\in T_{_G}$, since $g\in G$.  Hence the resulting semidirect product $\widehat G\cong S\rtimes G\leq \ddot G$ acts properly on $\mathbf R_N$, extending the cocompact action of $S\times T_{_G}$.
\end{proof}

Corollary~C from the Introduction follows easily: if $Z^m\times G$ is cocompactly cubulated for suitably chosen $m$, then $\integers^M\times(\integers^m\rtimes G)\cong\integers^{M+m}\rtimes G$ is cocompactly cubulated for all $M\geq 0$.

We conclude by solving Problem~16.10 of~\cite{WiseIsraelHierarchy}, in which Wise asked whether $\ddot G$ can be chosen, for any virtually-$\integers^n$ group $G$, to have the form $\integers^m\times G$.  This is not the case:

\begin{exmp}\label{exmp:trianglegroup}
Let $W\cong\langle a,b,c\mid[a,b],c^6,cac^{-1}=b,cbc^{-1}=a^{-1}b\rangle\cong\integers^2\rtimes\integers_6$.  As indicated at left in Figure~\ref{fig:hex2}, $W$ acts on the tiling of $\Euclidean^2$ by regular hexagons; the translations $a,b$ are along the illustrated vectors, and $c$ is the 6-fold rotation taking $a$ to $b$.

Suppose for some $m\geq 0$ that $\integers^m\times W$ acts properly and cocompactly on a CAT(0) cube complex $\mathbf X$.  Hence $\integers^m\times W$ acts with a single orbit of $(m+2)$-cubes on $\mathbf R_{m+2}$.  Identifying $\mathbf R_{m+2}$ with $\Euclidean^{m+2}$ and applying Bieberbach's theorem shows that the action of $\integers^m\times W$ on $\Euclidean^{m+2}$ preserves the tiling by $(m+2)$-cubes, and $W$ stabilizes a 2-dimensional subspace $\mathbb F$ and fixes pointwise an $m$-dimensional subspace $\mathbb P$ orthogonal to $\mathbb F$.  This induces an action of $P_{_W}=P_{\integers^m\times W}$ on an $(m+2)$-cube $C$ with an $m$-dimensional subspace $S$ fixed by $P_{_W}$.  The subspace $S$ is orthogonal to a hexagon $H\subset C$ that contains the origin and on which $c$ acts as a 6-fold rotation.  Thus $P_{_W}$ permutes the diagonals of $C$ orthogonal to $C$, so since $S$ is fixed, there is only one such diagonal.  Therefore $c$ acts as a 6-fold rotation of a 3-cube, which is impossible; every order-6 automorphism of a 3-cube is a rotation-reflection.

\begin{figure}[h]
  \includegraphics[width=0.9\textwidth]{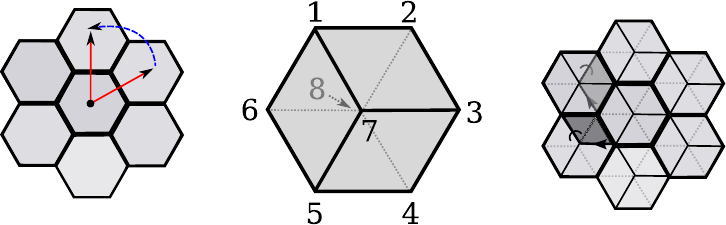}\\
  \caption{At left is the $W$-invariant tiling.  The central picture suggests the orientation-reversing $P_{_W}$-action on a 3-cube.  At right is part of $\mathbf R_3$, suggesting the $\integers\rtimes W$-action: $c$ takes one darkened square to the other and the $\integers$ semidirect factor acts as translation along the line through the point $x$ and the barycenter of the central cube.}\label{fig:hex2}
\end{figure}

By contrast, let $W$ act on $\integers$ by $a(1)=b(1)=1$ and $c(1)=-1$.  The resulting semidirect product $\integers\rtimes W$ acts properly and cocompactly on $\mathbf R_3$.  Indeed, as shown in the middle picture in Figure~\ref{fig:hex2}, $\integers_6$ acts as an orientation-reversing automorphism of a 3-cube -- in the notation of Figure~\ref{fig:hex2}, $c$ is the permutation $(123456)(78)$.  The plane containing the barycenters of the 3-cubes shown in Figure~\ref{fig:hex2} is stabilized by $c$, which acts as a 6-fold rotation.  The generator of the $\integers$ factor acts as a translation orthogonal to this plane, and $c$ acts on the axis of this translation by reversing signs.  Hence the action of $P_{_W}$ induces the given action of $\langle c\rangle$ on $\integers$.

Finally, Example~16.11 of~\cite{WiseIsraelHierarchy}, due to Dunbar, provides a torsion-free, 3-dimensional crystallographic group $D$ that is not cocompactly cubulated.  The group $D$ is obtained from the above presentation of $W$ by dropping the relation $c^6$.  Thus $D$ is virtually $\integers^3$, and $c$ acts as a ``screw motion'', translating and rotating by $\frac{\pi}{3}$ along an axis orthogonal to an $\langle a,b\rangle$-invariant plane.  In fact, no $\integers^m\times D$ is cocompactly cubulated: if $\integers^m\times D$ is cocompactly cubulated, then it acts with a single orbit of $(m+3)$-cubes on $\mathbf R_{m+3}$, and there is a $D$-invariant copy $\mathbb F$ of $\Euclidean^3$ containing a $\langle a,b\rangle$-invariant plane $\mathbb F_0$.  Moreover, $\mathbb F\cong \mathbb F_0\times\mathbb F_1$, where $\mathbb F_1$ is the screw-axis for $c$.  As before, this implies that the point group $P_{\integers^m\times D}\cong P_{_D}\cong\integers_6$ acts on an $(m+3)$-cube $C$, stabilizing a 3-dimensional subspace $H=C\cap\mathbb F$ and acting trivially on an $m$-dimensional subspace $S$.  Arguing as above, one now finds that the image of $c$ in $P_{_D}\cong\integers_6$, which acts trivially on the screw-axis, acts as a 6-fold rotation of a 3-cube, which is impossible.

Morally, adding dimensions only aids cocompact cubulation if the point group is allowed to act nontrivially, if necessary, on the group of translations in the new directions.
\end{exmp}

\bibliographystyle{alpha}
\bibliography{Bieberbachbib}
\end{document}